\documentclass{birkjour_t2}
\usepackage{amsmath}
\usepackage{tkz-euclide}
\usetkzobj{all}
\tikzset{elegant/.style={smooth,thick,samples=50,cyan}}
\tikzset{eaxis/.style={->,>=stealth}}
\usepackage{color,tikz}
\usetikzlibrary{decorations.pathreplacing,calc}
\usepackage{rotating}
\tikzset{liltext/.style={font=\tiny}}
\usepackage{amsfonts}
\usepackage{amssymb}
\usepackage[english]{babel}
\usepackage{color}
\usepackage{graphicx}
\usepackage{lmodern}
\usepackage{amsthm}
\usepackage{mathrsfs}
\usepackage{microtype}
\usepackage{mathscinet}
\usepackage{enumitem}
\usepackage[cal=boondoxo,bb=ams]{mathalfa}
\usepackage{hyperref}
\hypersetup{hidelinks}

\renewcommand{\theequation}{\arabic{section}.\arabic{equation}}

\newtheorem{defn}{Definition}[section]
\newtheorem{thm}{Theorem}[section]
\newtheorem{prop}{Proposition}[section]
\newtheorem{lem}{Lemma}[section]

\newtheorem{rem}{Remark}[section]

\newcommand{\ml}{\mathcal}
\newcommand{\mb}{\mathbb}

\DeclareMathOperator{\supp}{supp}

\begin{document}
	
	%
	%
	%
	%
	%
	%
	%
	%

	\title[Semilinear Moore -- Gibson -- Thompson equation]
	{Nonexistence of global solutions for the semilinear Moore -- Gibson -- Thompson equation in the conservative case}

	\author[W. Chen]{Wenhui Chen}
	\address{Institute of Applied Analysis, Faculty of Mathematics and Computer Science\\
		Technical University Bergakademie Freiberg\\
		Pr\"{u}ferstra{\ss}e 9\\
		09596 Freiberg\\
		Germany}
	\email{wenhui.chen.math@gmail.com}
	
	\author[A. Palmieri]{Alessandro Palmieri}
	
	\address{Department of Mathematics\\
		University of Pisa\\
		Largo B. Pontecorvo 5\\
		56127 Pisa\\
		Italy}
	\email{alessandro.palmieri.math@gmail.com}
	
	
	
	\begin{abstract} 
In this work, the Cauchy problem for the semilinear Moore -- Gibson -- Thompson (MGT) equation with power nonlinearity $|u|^p$ on the right -- hand side is studied. Applying $L^2$ -- $L^2$ estimates and a fixed point theorem, we obtain local (in time) existence of solutions to the semilinear MGT equation. Then, the blow -- up of local in time solutions is proved by using an iteration method, under certain sign assumption for initial data, and providing that the exponent of the power of the nonlinearity fulfills $1<p\leqslant p_{\mathrm{Str}}(n)$ for $n\geqslant2$ and $p>1$ for $n=1$. Here the Strauss exponent $p_{\mathrm{Str}}(n)$ is the critical exponent for the semilinear wave equation with power nonlinearity. In particular, in the limit case $p=p_{\mathrm{Str}}(n)$ a different approach with a weighted space average of a local in time solution is considered.
	\end{abstract}
	
	\maketitle
\section{Introduction}
In recent years, the Moore - Gibson - Thompson (MGT) equation, a linearization of a model for wave propagation in viscous thermally relaxing fluids, has caught a lot of attention (see  \cite{MooreGibson1960,Thompson1972,Gorain2010,KaltenbacherLasieckaMarchand2011,
	KaltenbacherLasiecka2012,MarchandMcDevittTriggiani2012,Jordan2014,
	LasieckaWang2015,PellicerSola2015,CaixetaLasieckaDominos2016,
	DellOroLasieckaPata2016,LasieckaWang2016,DellOroPata2017,
	Lasiecka2017,PellicerSaiHouari2017,AlvesCaixetaSilvaRodrigues2018,
	DellOroLasieckaPata2019} and references therein). This model is realized through the third order hyperbolic partial differential equation 
\begin{equation}\label{General.MGT.Equation}
\tau u_{ttt}+ u_{tt}-c^2\Delta u-b\Delta u_t=0.
\end{equation} 
In the physical context of acoustic waves, the unknown function $u=u(t,x)$ denotes a scalar acoustic velocity, $c$ denotes the speed of sound and $\tau$ denotes the thermal relaxation. Besides, the coefficient $b=\beta c^2$ is related to the diffusivity of the sound with {$\tau\in (0,\beta]$}.  In particular, there is a transition from a linear model  that can be described with an exponentially stable strongly continuous semigroup in the case {$0<\tau<\beta$} to the limit case $\beta=\tau$, where the exponential stability of a semigroup is lost and it holds the conservation of a suitable defined energy (see \cite{KaltenbacherLasieckaMarchand2011,MarchandMcDevittTriggiani2012}). For this reason, we shall call the limit case $\beta=\tau$ the \emph{conservative case}. 

In this paper, we consider the semilinear Cauchy problem associated to the MGT equation 
\begin{equation}\label{Semi.MGT.Equation}
\begin{cases}
\beta u_{ttt}+u_{tt}-\Delta u-\beta \Delta u_t=|u|^p,&x\in\mb{R}^n,\,t>0,\\
(u,u_t,u_{tt})(0,x)=(u_0,u_1,u_2)(x),&x\in\mb{R}^n,
\end{cases}
\end{equation}
in the limit case $\tau=\beta>0$, where $p>1$ and, for the sake of simplicity, we normalized the speed of the sound by putting $c^2=1$. We are interested to the 
blow -- up in finite time of local (in time) solutions under suitable sign assumptions for the Cauchy data regardless of their size and for suitable values of the exponent $p$.

Let us recall some results that are related to our model \eqref{Semi.MGT.Equation}. By taking formally $\beta=0$, we find the semilinear wave equation
\begin{equation}\label{Semi.Wave.Eq}
\begin{cases}
u_{tt}-\Delta u=|u|^p,& x\in\mb{R}^n,\,t>0,\\
(u,u_t)(0,x)=(u_0,u_1)(x),&x\in\mb{R}^n,
\end{cases}
\end{equation}
where $p>1$. According to  \cite{Joh79,Kato80,Str81,Gla81-g,Gla81-b,Sid84,Sch85,Zhou95,LS96,GLS97,Tataru2001,Jiao03,YordanovZhang2006,Zhou07,LaiZhou14} the so -- called \emph{Strauss exponent} $p_{\mathrm{Str}}(n)$ is the critical exponent of \eqref{Semi.Wave.Eq}, where  $p_{\mathrm{Str}}(n)$ is the positive root of the quadratic equation 
\begin{align}\label{Eq.Stauss.Exponent}
(n-1)p^2-(n+1)p-2=0.
\end{align}
The lifespan of solutions  to \eqref{Semi.Wave.Eq} that blow up in finite time has been intensively considered. Here we refer to \cite{Lindblad1990,Zhou199201,Zhou199202,Zhou1993,LS96, DiGeorgiev2001,TakamuraWakasa2011,ZhouHan2014, Takamura2015,Takamura2019}. According to these works, the  sharp estimates for the lifespan $T(\varepsilon)$  are given by 
\begin{align*}
T(\varepsilon) \approx  \begin{cases}  C \varepsilon^{- \frac{2p(p-1)}{2+(n+1)p-(n-1)p^2}} & \mbox{if} \ 1<p<p_{\mathrm{Str}}(n),  \\ \exp\big(C \varepsilon^{-p(p-1)}\big)  & \mbox{if} \ p=p_{\mathrm{Str}}(n)  ,
\end{cases} 
\end{align*} for $n\geqslant 3$ and for $n=2$ and $2<p<p_{\mathrm{Str}}(2)$,  by $$T(\varepsilon)\approx \begin{cases} \varepsilon^{-\frac{p-1}{3-p}} & \mbox{if} \ 1<p<2, \\ a(\varepsilon) & \mbox{if} \ p=2, \end{cases}$$  for $n=2$ and $1<p\leqslant 2$, where $a=a(\varepsilon)$ satisfies $a^2\varepsilon^2\log(1+a)=1$, and by $$T(\varepsilon) \approx \varepsilon^{-\frac{p-1}{2}}$$  for $n=1$ and $p>1$,  where in each case $\varepsilon$ is a sufficiently small positive quantity. Note that, for the sake of simplicity, in the previous lifespan estimates for low dimensions $n=1,2$ we restricted our considerations to the case in which the integral of $u_1$ is not zero.



Our main results Theorem \ref{Thm blow-up sub Strauss} and Theorem \ref{Thm critical case}, which are stated and proved in Section \ref{Section blow up} and in Section \ref{Section blow up critical case}, respectively, are blow -- up results for the semilinear model \eqref{Semi.MGT.Equation} that hold for exponents of the nonlinearity such that $1<p\leqslant p_{\mathrm{Str}}(n)$ and under suitable sign assumptions for compactly supported initial data. Furthermore, we will obtain an upper bound estimate for the lifespan of local solutions to \eqref{Semi.MGT.Equation} which coincides in some cases with the optimal one for \eqref{Semi.Wave.Eq}, as we have just recalled. The proof of Theorem  \ref{Thm blow-up sub Strauss} is based on an iteration argument, which allows us to show the blow -- up in finite time of the space average of a local in time solution to \eqref{Semi.MGT.Equation}, while in Theorem \ref{Thm critical case} a weighted version of the space average is employed as time -- dependent functional.

Let us point out that in the subcritical case, i.e. for $1<p<p_{\mathrm{Str}}(n)$, the iteration argument is not just a straightforward generalization of the one for \eqref{Semi.Wave.Eq}. Indeed, in the iteration procedure we have to deal with an unbounded exponential multiplier. For this purpose, we propose a slicing procedure of the domain of integration by taking inspiration from \cite{AKT00}, even though the sequence of the parameters (cf. $\{L_j\}_{j\in\mathbb{N}}$ below in Subsection \ref{Subsection Iter Arg}), that characterize the slicing of the domain of integration, has a  quite different structure. Up to our best knowledge, our result is the first attempt to include an unbounded exponential multiplier in an iteration argument for proving a blow -- up result for hyperbolic semilinear models.

In the critical case, i.e. for $p=p_{\mathrm{Str}}(n)$ and $n\geqslant 2$, the approach with the space average of the solution is no longer suitable and it has to be refined. This is done by considering a weighted space average (with a weighted function depending on the time variable as well) as functional, whose dynamic is studied in the iteration procedure. In this case, we follow the approach developed in \cite{WakYor18}, which is based on the  so -- called \emph{slicing method}, developed for the first time in \cite{AKT00}. Nonetheless, as in the subcritical case, we have to consider a sequence of parameters (cf. $\{\Omega_j\}_{j\in\mathbb{N}}$ in Section \ref{Section blow up critical case}) which characterize the slicing procedure of the domain of integration that has a relatively different structure with respect to the one which usually used to deal with critical cases (see for example \cite{AKT00,WakYor18,PalTak19,PalTak19dt,PalTak19mix,PalTu19}).

The present paper is organized as follows. In Section \ref{Section Linear problem} we first derive $L^2$ -- $L^2$ estimates and well -- posedness for the linear MGT equation. In Section \ref{Section local existence} combining Banach fixed point theorem with the derived $L^2$ -- $L^2$ estimates, the local (in time) existence of solutions to the semilinear MGT equation is proved. Then, in Section \ref{Section blow up} we apply an iteration method associated with the test function introduced in \cite{YordanovZhang2006} to prove the blow -- up of energy solutions in the subcritical case. Afterwards, in Section \ref{Section blow up critical case} we prove the blow -- up of a local in time solution (under certain assumptions for the initial data) also in the critical case $p=p_{\mathrm{Str}}(n)$ when $n\geqslant 2$. Finally, some concluding remarks in Section \ref{Section Concluding remarks} complete the paper.


\bigskip

\noindent\textbf{Notation: } We give some notations to be used in this paper. We write $f\lesssim g$ when there exists a positive constant $C$ such that $f\leqslant Cg$. We denote $g\lesssim f \lesssim g$ by $f\approx g$. 
Moreover, $B_R$ denotes the ball around the origin with radius $R$ in $\mathbb{R}^n$. As mentioned in the introduction, $p_{\mathrm{Str}}(n)$ denotes the Strauss exponent.

\section{Linear problem for the MGT equation} \label{Section Linear problem}
In this section, we will derive some qualitative properties of solutions to the corresponding linearized Cauchy problem to \eqref{Semi.MGT.Equation}, which is advantageous for us in order to understand the semilinear problem. More precisely, we are interested in the following linear MGT equation:
\begin{equation}\label{Linear.MGT.Equation}
\begin{cases}
\beta u_{ttt}+u_{tt}-\Delta u-\beta \Delta u_t=0,&x\in\mb{R}^n,\,t>0,\\
(u,u_t,u_{tt})(0,x)=(u_0,u_1,u_2)(x),&x\in\mb{R}^n,
\end{cases}
\end{equation}
where $\beta$ is a positive constant.
\begin{rem}
	The principal symbol of the equation in \eqref{Linear.MGT.Equation} is given by $\tau^3-\tau|\eta|^2=0$ and has real and pairwise distinct roots $\tau=0$, $\tau=|\eta|$ and $\tau=-|\eta|$. Thus, the linear MGT equation in \eqref{Linear.MGT.Equation} is strictly hyperbolic.
\end{rem}

According to \cite{KaltenbacherLasieckaMarchand2011,LasieckaWang2016}, a suitable defined energy for the MGT equation is conserved from the point of view of semigroups. 

We now state the energy conservation result for the linear homogeneous Cauchy problem.
\begin{prop}\label{Prop.Con.Energy}
	Let us introduce the following energy for a solution $u$ to \eqref{Linear.MGT.Equation}:
	\begin{align*}
	E_{\emph{MGT}}[u](t)\doteq\tfrac{1}{2}\|\partial_t(\beta u_{t}+u)(t,\cdot)\|_{L^2(\mb{R}^n)}^2+\tfrac{1}{2}\|\nabla_x(\beta u_t+u)(t,\cdot)\|_{L^2(\mb{R}^n)}^2.
	\end{align*}
	Then, this energy is conserved, i.e., $E_{\emph{MGT}}[u](t)\equiv E_{\emph{MGT}}[u](0)$ for any $t>0$.
\end{prop}
\begin{proof}
	Indeed, the linear MGT equation in \eqref{Linear.MGT.Equation} can be rewritten as
	\begin{align*}
	\partial_t^2(\beta u_t+u)-\Delta(\beta u_t+u)=0,
	\end{align*}
	which implies that the unknown function $\beta u_t+u$ is the solution to free wave equation. From the energy conservation for the free wave equation, we immediately complete the proof.
\end{proof}

\begin{prop}\label{Prop.Well-Posed}
	Let $n\geqslant 1$. Let us consider $(u_0,u_1,u_2)\in H^2(\mb{R}^n)\times H^1(\mb{R}^n)\times L^2(\mb{R}^n)$. Then, there exists a uniquely determined  solution
	\begin{align*}
	u\in\ml{C}([0,T],H^2(\mb{R}^n))\cap \ml{C}^1([0,T],H^1(\mb{R}^n))\cap \ml{C}^2([0,T],L^2(\mb{R}^n))
	\end{align*}
	to \eqref{Linear.MGT.Equation} for all $T>0$. Moreover, the solution to \eqref{Linear.MGT.Equation} satisfies the following estimates for $\ell=0,1$:
	\begin{align*}
	\|u(t,\cdot)\|_{L^2(\mb{R}^n)}&\lesssim\|u_0\|_{L^2(\mb{R}^n)}+(1+t)\left(\|u_1\|_{L^2(\mb{R}^n)}+\|u_2\|_{L^2(\mb{R}^n)}\right),\\
	\|\nabla_x^{1+\ell}u(t,\cdot)\|_{L^2(\mb{R}^n)}&\lesssim\|u_0\|_{H^{1+\ell}(\mb{R}^n)}+\|u_1\|_{H^{\ell}(\mb{R}^n)}+\|u_2\|_{L^2(\mb{R}^n)},\\
	\|\nabla_x^{\ell}u_t(t,\cdot)\|_{L^2(\mb{R}^n)}&\lesssim\|u_0\|_{H^{\ell}(\mb{R}^n)}+\|u_1\|_{H^{\ell}(\mb{R}^n)}+\|u_2\|_{L^2(\mb{R}^n)},\\
	\|u_{tt}(t,\cdot)\|_{L^2(\mb{R}^n)}&\lesssim\|u_0\|_{H^1(\mb{R}^n)}+\|u_1\|_{H^1(\mb{R}^n)}+\|u_2\|_{L^2(\mb{R}^n)}.
	\end{align*}
\end{prop}
\begin{proof}
	Employing the partial Fourier transform with respect to spatial variables to \eqref{Linear.MGT.Equation}, we get
	\begin{equation}\label{Eq.MGTE.F}
	\begin{cases}
	\beta\hat{u}_{ttt}+\hat{u}_{tt}+|\xi|^2 \hat{u}+\beta|\xi|^2 \hat{u}_t=0,&\xi\in\mb{R}^n,\,t>0,\\
	(\hat{u},\hat{u}_t,\hat{u}_{tt})(0,\xi)=(\hat{u}_0,\hat{u}_1,\hat{u}_2)(\xi),&\xi\in\mb{R}^n.
	\end{cases}
	\end{equation}
	By direct calculations, the characteristic roots of \eqref{Eq.MGTE.F} are
	\begin{align*}
	\lambda_{1,2}=\pm i|\xi|\quad\text{and}\quad \lambda_3=-\tfrac{1}{\beta}.
	\end{align*}
	Therefore, the solution to \eqref{Eq.MGTE.F} is given by
	\begin{align*}
	\hat{u}(t,\xi)
	&=\left(\tfrac{\cos(|\xi|t)}{1+\beta^2|\xi|^2}+\tfrac{\beta|\xi|\sin(|\xi|t)}{1+\beta^2|\xi|^2}+\tfrac{\beta^2|\xi|^2}{2(1+\beta^2|\xi|^2)}\mathrm{e}^{-t/\beta}\right)\hat{u}_0(\xi)+\tfrac{\sin(|\xi|t)}{|\xi|}\hat{u}_1(\xi)\\
	&\quad+\left(\tfrac{\beta\sin(|\xi|t)}{|\xi|(1+\beta^2|\xi|^2)}-\tfrac{\beta^2\cos(|\xi|t)}{1+\beta^2|\xi|^2}+\tfrac{\beta^2}{2(1+\beta^2|\xi|^2)}\mathrm{e}^{-t/\beta}\right)\hat{u}_2(\xi).
	\end{align*}

	By applying the same approach used to prove the existence of solutions in the classical energy space to the Cauchy problem for free wave equation (e.g. Chapter 14 in \cite{EbertReissig2018}) and the mean value theorem, we may conclude the existence of solutions to the MGT equation \eqref{Linear.MGT.Equation}.
	
	The conservation of the energy stated in Proposition \ref{Prop.Con.Energy} leads immediately to the uniqueness of the solution to the Cauchy problem \eqref{Linear.MGT.Equation}.
	
	Finally, in order to get the desired estimates of solutions, we apply $|\sin(|\xi|t)|\leqslant |\xi|t$ for $|\xi|\leqslant\epsilon\ll1$, $|\sin(|\xi|t)|\leqslant 1$ and $|\cos(|\xi|t)|\leqslant 1$. Thus, the proof is completed.
\end{proof}

To conclude this section, we point out that the solution  to the linear Cauchy problem for MGT equation \eqref{Linear.MGT.Equation} fulfills the inhomogeneous wave equation
\begin{equation}\label{MGT.Equation.Wave}
\begin{cases}
u_{tt}-\Delta u=\mathrm{e}^{-t/\beta}(u_2(x)-\Delta u_0(x)),&x\in\mb{R}^n,\,t>0,\\
(u,u_t)(0,x)=(u_0,u_1)(x),&x\in\mb{R}^n.
\end{cases}
\end{equation}
Thus, we claim that $\mathrm{supp } \,u(t,\cdot)\subset B_{R+t}$, if we assume $\mathrm{supp } \, u_j\subset B_R$ for any $j=0,1,2$ and for some $R>0$. Indeed, the source $f(t,x)=\mathrm{e}^{-t/\beta}(u_2(x)-\Delta u_0(x))$ has support contained in the forward cone $\{(t,x): |x|\leqslant R+t\}$ under these assumptions and we can use the property of finite speed of propagation for the classical wave equation.

\section{Existence of local (in time) solution}\label{Section local existence}
\begin{thm}\label{Thm.Local.Exist}
	Let $n\geqslant 1$. Let us consider $(u_0,u_1,u_2)\in H^2(\mb{R}^n)\times H^1(\mb{R}^n)\times L^2(\mb{R}^n)$ compactly supported with  $\mathrm{supp } \, u_j\subset B_R$ for any $j=0,1,2$ and for some $R>0$. We assume $p>1$ such that $ p\leqslant n/(n-2)$ when $n\geqslant 3$. Then, there exists a positive $T$ and a uniquely determined local (in time) mild solution 
	\begin{align*}
	u\in\ml{C}([0,T],H^2(\mb{R}^n))\cap \ml{C}^1([0,T],H^1(\mb{R}^n))\cap \ml{C}^2([0,T],L^2(\mb{R}^n))
	\end{align*} to \eqref{Semi.MGT.Equation} satisfying $\mathrm{supp }\,u(t,\cdot)\subset B_{R+t}$ for any $t \in [0,T]$.
\end{thm}
Let us introduce some notations for the proof of the local (in time) existence of solutions. We denote by $K_0(t,x)$, $K_1(t,x)$ and $K_2(t,x)$ the fundamental solutions to the  linear Cauchy problem \eqref{Linear.MGT.Equation}
with initial data $(u_0,u_1,u_2)=(\delta_0,0,0)$, $(u_0,u_1,u_2)=(0,\delta_0,0)$ and $(u_0,u_1,u_2)=(0,0,\delta_0)$, respectively. Here $\delta_0$ is the Dirac distribution in $x=0$ with respect to spatial variables. 
Therefore, the solution to \eqref{Linear.MGT.Equation} is given by
\begin{align*}
u(t,x)=K_0(t,x)\ast_{(x)}u_0(x)+K_1(t,x)\ast_{(x)}u_1(x)+K_2(t,x)\ast_{(x)}u_2(x),
\end{align*}
where the Fourier transforms of the kernels $K_0(t,x)$, $K_1(t,x)$ and $K_2(t,x)$ are given by
\begin{align*}
\widehat{K_0}(t,\xi)&=\tfrac{\cos(|\xi|t)}{1+\beta^2|\xi|^2}+\tfrac{\beta|\xi|\sin(|\xi|t)}{1+\beta^2|\xi|^2}+\tfrac{\beta^2|\xi|^2}{2(1+\beta^2|\xi|^2)}\mathrm{e}^{-t/\beta}, \\ \widehat{K_1}(t,\xi) &=\tfrac{\sin(|\xi|t)}{|\xi|},\\
\widehat{K_2}(t,\xi)&=\tfrac{\beta\sin(|\xi|t)}{|\xi|(1+\beta^2|\xi|^2)}-\tfrac{\beta^2\cos(|\xi|t)}{1+\beta^2|\xi|^2}+\tfrac{\beta^2}{2(1+\beta^2|\xi|^2)}\mathrm{e}^{-t/\beta}.
\end{align*}
\begin{proof}
	Let us define the family of evolution spaces
	\begin{align*}
	X(T)\doteq&\big\{u\in\ml{C}([0,T],H^2(\mb{R}^n))\cap \ml{C}^1([0,T],H^1(\mb{R}^n))\cap \ml{C}^2([0,T],L^2(\mb{R}^n))\\ &\,\,\,\,\text{with }\mathrm{supp }\,u(t,\cdot)\subset B_{R+t} \ \mbox{for any} \ t\in [0,T]\big\},
	\end{align*}
	with the norm
	\begin{align*}
	\|u\|_{X(T)}\doteq\max\limits_{t\in[0,T]}\sum\limits_{\ell+j\leqslant 2,\,\,\ell,j\in\mb{N}_0}\left\| \nabla_x^{\ell}\partial_t^ju(t,\cdot)\right\|_{L^2(\mb{R}^n)}.
	\end{align*}
	According to Duhamel's principle, we introduce the operator
	\begin{align*}
	N:\,\,u\in X(T)\rightarrow Nu(t,x) & \doteq  K_0(t,x)\ast_{(x)}u_0(x)+K_1(t,x)\ast_{(x)}u_1(x)+K_2(t,x)\ast_{(x)}u_2(x)\\
	& \quad +\int_0^t K_2(t-\tau,x)\ast_{(x)}|u(\tau,x)|^p\mathrm{d}\tau.
	\end{align*} We will consider as mild local in time solutions to \eqref{Semi.MGT.Equation} the fixed points of the operator $N$.
	Therefore, with the aim of deriving the local (in time) existence and uniqueness of the  solution in $X(T)$, we need to prove
	\begin{align}
	\|Nu\|_{X(T)}&\leqslant C_0(u_0,u_1,u_2)+C_1(u_0,u_1,u_2)\,T\|u\|_{X(T)}^p,\label{IMP.01}\\
	\|Nu-Nv\|_{X(T)}&\leqslant C_2(u_0,u_1,u_2)\, T\|u-v\|_{X(T)}\left(\|u\|_{X(T)}^{p-1}+\|v\|_{X(T)}^{p-1}\right).\label{IMP.02}
	\end{align}
	
	First of all, from Proposition \ref{Prop.Well-Posed} it is clear that
	\begin{align*}
	u^{\mathrm{ln}}(t,x)\doteq K_0(t,x)\ast_{(x)}u_0(x)+K_1(t,x)\ast_{(x)}u_1(x)+K_2(t,x)\ast_{(x)}u_2(x)\in  X(T)
	\end{align*} and 
	\begin{align*}
	\| u^{\mathrm{ln}}\|_{X(T)} \lesssim \|u_0\|_{H^2(\mathbb{R}^n)}+ (1+T) \left(\|u_1\|_{H^1(\mathbb{R}^n)}+\|u_2\|_{L^2(\mathbb{R}^n)}\right).
	\end{align*}
	
	Next, to prove \eqref{IMP.01}, we apply the classical Gagliardo -- Nirenberg inequality.
	Thus, we get for any $\tau\in[0,T]$
	\begin{align*}
	\|u(\tau,\cdot)\|_{L^{2p}(\mb{R}^n)}^p\leqslant C\|u(\tau,\cdot)\|_{L^2(\mb{R}^n)}^{(1-\frac{n}{2}(1-\frac{1}{p}))p}\|u(\tau,\cdot)\|_{\dot{H}^1(\mb{R}^n)}^{\frac{n}{2}(1-\frac{1}{p})p}\leqslant C\|u\|_{X(T)}^p,
	\end{align*}
	where $p> 1$ if $n=1,2$ and $1< p\leqslant n/(n-2)$ if $n\geqslant3$.
	
	Then, applying the previous inequality and  using $L^2$ -- $L^2$ estimates from Proposition \ref{Prop.Well-Posed}, we derive
	\begin{align*}
	\left\|\int_0^t K_2(t-\tau,x)\ast_{(x)}|u(\tau,x)|^p\mathrm{d}\tau\right\|_{L^2(\mb{R}^n)}&\leqslant C \int_0^t(1+t-\tau)\||u(\tau,\cdot)|^p\|_{L^2(\mb{R}^n)}\mathrm{d}\tau\\
	&\leqslant C(1+t)t\|u\|_{X(T)}^p.
	\end{align*}
	Analogously,
	\begin{align*}
	\left\|\nabla_x^{\ell}\partial_t^j\int_0^t K_2(t-\tau,x)\ast_{(x)}|u(\tau,x)|^p\mathrm{d}\tau\right\|_{L^2(\mb{R}^n)}&\leqslant C \int_0^t\||u(\tau,\cdot)|^p\|_{L^2(\mb{R}^n)}\mathrm{d}\tau\\
	&\leqslant Ct\|u\|_{X(T)}^p,
	\end{align*}
	for any $\ell,j\in\mb{N}_0$ such that $1\leqslant\ell+j\leqslant2$.
	
	Finally, $Nu$ satisfies the support condition $\mathrm{supp}\, Nu(t,\cdot) \subset B_{R+t}$ for any $t\in[0,T]$, since $w=Nu$ is a solution of the inhomogeneous Cauchy problem for the wave equation
	\begin{align*}
	\begin{cases} w_{tt} -\Delta w = \mathrm{e}^{-t/\beta}(u_2(x)-\Delta u_0(x)) +\displaystyle{ \tfrac{1}{\beta}\int_{0}^{t} \mathrm{e}^{(\tau -t)/\beta}|u(\tau,x)|^p \mathrm{d}\tau}, & x\in\mathbb{R}^n,\, t>0, \\
	(w,w_t)(0,x) = (u_0,u_1)(x), &   x\in\mathbb{R}^n,
	\end{cases}
	\end{align*} and $u$ is supported in the forward cone due to $u\in X(T)$.
	Thus, we may conclude that  $N$ maps $X(T)$ into itself and  \eqref{IMP.01}. 
	
	To derive \eqref{IMP.02}, we remark that
	\begin{align*}
	\|Nu-Nv\|_{X(t)}=\left\|\int_0^tK_2(t-\tau,x)\ast_{(x)}\left(|u(\tau,x)|^p-|v(\tau,x)|^p\right)\mathrm{d}\tau\right\|_{X(t)}.
	\end{align*}
	By employing
	\begin{align*}
	\left||u(\tau,x)|^p-|v(\tau,x)|^p\right|\leqslant C |u(\tau,x)-v(\tau,x)| \left(|u(\tau,x)|^{p-1}+|v(\tau,x)|^{p-1}\right)
	\end{align*}
	and H\"older's inequality, we conclude
	\begin{align*}
	\left\||u(\tau,\cdot)|^p-|v(\tau,\cdot)|^p\right\|_{L^2(\mb{R}^n)}\leqslant C\|u(\tau,\cdot)-v(\tau,\cdot)\|_{L^{2p}(\mb{R}^n)}\left(\|u(\tau,\cdot)\|_{L^{2p}(\mb{R}^n)}^{p-1}+\|v(\tau,\cdot)\|_{L^{2p}(\mb{R}^n)}^{p-1}\right).
	\end{align*}
	Finally, by using $L^2$ -- $L^2$ estimates from Proposition \ref{Prop.Well-Posed} again, we immediately obtain the desired estimate \eqref{IMP.02}.  This completes the proof.
	
\end{proof}


\section{Blow -- up result in the subcritical case 
}\label{Section blow up}

Let $u=u(t,x)$ be a local in time solution to the semilinear Cauchy problem 
\begin{equation}\label{Semi.MGT.Equation.epsilon}
\begin{cases}
\beta u_{ttt}+u_{tt}-\Delta u-\beta \Delta u_t=|u|^p, & x\in\mb{R}^n,\, t>0, \\
(u,u_t,u_{tt})(0,x)= \varepsilon ( u_0, u_1, u_2)(x), &x\in\mb{R}^n,
\end{cases}
\end{equation} where $\varepsilon>0$ is a parameter describing the smallness of initial data.

The aim of this section is to prove the blow -- up of local (in time) solutions to \eqref{Semi.MGT.Equation.epsilon} in the subcritical case, that is for $1<p<p_{\mathrm{Str}}(n)$, under suitable conditions for the Cauchy data, and to derive an upper bound estimate for the lifespan. To do this, we introduce first the definition of energy solutions to \eqref{Semi.MGT.Equation.epsilon}.

%
%

\begin{defn} \label{Definition Energy solution}
	Let $(u_0,u_1,u_2)\in H^2(\mb{R}^n)\times H^1(\mb{R}^n)\times L^2(\mb{R}^n)$. We say $u$ is an energy solution to \eqref{Semi.MGT.Equation.epsilon} on $[0,T)$ if
	\begin{align*}
	u\in\ml{C}([0,T),H^2(\mb{R}^n))\cap \ml{C}^1([0,T),H^1(\mb{R}^n))\cap \ml{C}^2([0,T),L^2(\mb{R}^n))\cap L^p_{\emph{loc}}([0,T)\times\mb{R}^n)
	\end{align*}
	satisfy $u(0,\cdot)=\varepsilon u_0$ in $H^2(\mb{R}^n)$ and the integral identity
	\begin{align}\label{Def.Energy solution 1st}
	\beta\int_{\mb{R}^n}&u_{tt}(t,x)\psi(t,x)\, \mathrm{d}x+\int_{\mb{R}^n}u_t(t,x)\psi(t,x)\,\mathrm{d}x-\beta \varepsilon\int_{\mb{R}^n} u_2(x)\psi(0,x)\,\mathrm{d}x-\varepsilon\int_{\mb{R}^n} u_1(x)\psi(0,x)\, \mathrm{d}x\notag\\
	&+\beta\int_0^t\int_{\mb{R}^n}\big(\nabla_x u_t(s,x)\cdot\nabla_x\psi(s,x)-u_{tt}(s,x){\psi_s(s,x)}\big)\,\mathrm{d}x\, \mathrm{d}s\notag\\
	&+\int_0^t\int_{\mb{R}^n}\big(\nabla_x u(s,x)\cdot\nabla_x\psi(s,x)-u_t(s,x){\psi_s(s,x)}\big)\,\mathrm{d}x\,\mathrm{d}s\notag\\ &=\int_0^t\int_{\mb{R}^n}|u(s,x)|^p\psi(s,x)\,\mathrm{d}x\,\mathrm{d}s
	\end{align}
	for any $\psi\in\ml{C}_0^{\infty}([0,T)\times\mb{R}^n)$ and any $t\in[0,T)$.
\end{defn}

Applying a further step of integration by parts in \eqref{Def.Energy solution 1st}, it results 
\begin{align}
\int_0^t & \int_{\mb{R}^n} \big( -\beta\psi_{sss}(s,x)+\psi_{ss}(s,x)-\Delta\psi(s,x)+\beta \Delta \psi_{s}(s,x) \big) u(s,x)\,  \mathrm{d}x \, \mathrm{d}s \notag\\
& + \beta \int_{\mb{R}^n} \big( \psi(t,x) u_{tt}(t,x)-\psi_{t}(t,x)u_{t}(t,x)+\psi_{tt}(t,x)u(t,x)- \Delta \psi(t,x)u(t,x) \big)  \mathrm{d}x \notag \\
& - \beta \varepsilon \int_{\mb{R}^n} \big( \psi(0,x) u_{2}(x)-\psi_{t}(0,x)u_{1}(x)+\psi_{tt}(0,x)u_0(x)- \Delta \psi(0,x)u_0(x) \big)  \mathrm{d}x \notag\\
& + \int_{\mb{R}^n} \big( \psi(t,x) u_{t}(t,x)-\psi_{t}(t,x)u(t,x) \big)  \mathrm{d}x -\varepsilon \int_{\mb{R}^n} \big( \psi(0,x) u_{1}(x)-\psi_{t}(0,x)u_0(x) \big)  \mathrm{d}x \notag\\
&=\int_0^t\int_{\mb{R}^n}|u(s,x)|^p\psi(s,x) \,  \mathrm{d}x \, \mathrm{d}s.  \label{Def.Energy solution}
\end{align}
In particular, letting $t\rightarrow T$, we find that $u$ fulfills the definition of weak solution to \eqref{Semi.MGT.Equation.epsilon}.

%

\begin{thm} \label{Thm blow-up sub Strauss} Let us consider $p>1$ such that 
	\begin{align*}
	\begin{cases} p<  \infty & \mbox{if} \ \  n=1, \\ p< p_{\mathrm{Str}}(n) & \mbox{if} \ \ n\geqslant 2. \end{cases}
	\end{align*}
	Let $(u_0,u_1,u_2)\in H^2(\mathbb{R}^n)\times H^1(\mathbb{R}^n)\times L^2(\mathbb{R}^n)$ be nonnegative and compactly supported functions with supports contained in $B_R$ for some $R>0$ such that $u_0$ is not identically zero. Let \begin{align*}
	u\in\ml{C}([0,T),H^2(\mb{R}^n))\cap \ml{C}^1([0,T),H^1(\mb{R}^n))\cap \ml{C}^2([0,T),L^2(\mb{R}^n))\cap L^p_{\emph{loc}}([0,T)\times\mb{R}^n)
	\end{align*} be an energy solution on $[0,T)$ to the Cauchy problem \eqref{Semi.MGT.Equation.epsilon}  according to Definition \ref{Definition Energy solution} with lifespan $T=T(\varepsilon)$ such that
	\begin{align}\label{support condition u}
	\mathrm{supp}\, u(t,\cdot) \subset B_{R+t} \quad \mbox{for any} \ t\in (0,T).
	\end{align} Then, there exists a positive constant $\varepsilon_0=\varepsilon_0(u_0,u_1,u_2,n,p,R,\beta)$ such that for any $\varepsilon\in(0,\varepsilon_0]$ the solution $u$ blows up in finite time. Furthermore, the upper bound estimate for the lifespan
	\begin{align*}
	T(\varepsilon)\leqslant C \varepsilon^{-\frac{2p(p-1)}{\theta(p,n)}}
	\end{align*} holds, where $C$ is an independent of $\varepsilon$, positive constant and 
	\begin{align}\label{def theta(p,n)}
	\theta(p,n)\doteq 2+(n+1)p-(n-1)p^2.
	\end{align}
\end{thm}

\subsection{Iteration frame}

According to Theorem \ref{Thm blow-up sub Strauss}, we assume that $u_0,u_1$ and $u_2$ are nonnegative functions, with nontrivial $u_0$, and compactly supported with support contained in $B_R$ for some suitable $R>0$.


Then, thanks to what we underlined in Section \ref{Section local existence}, we have
\begin{align}\label{property of finite speed of propagation u(t,.)}
\mathrm{supp} \, u(t,\cdot) \subset B_{R+t} \qquad  \mbox{for any} \ t\in (0,T).
\end{align}
We introduce now the following time -- dependent functional:
\begin{align*}
U(t)\doteq \int_{\mathbb{R}^n} u(t,x) \, \mathrm{d} x. 
\end{align*} 
Choosing a test function $\psi$ in \eqref{Def.Energy solution 1st} such that $\psi=1$ on $\{(s,x)\in [0,t]\times\mathbb{R}^n: |x|\leqslant R+ s\}$, due to \eqref{property of finite speed of propagation u(t,.)} we have 
\begin{align*}
\beta\int_{\mb{R}^n}&u_{tt}(t,x)\, \mathrm{d}x+\int_{\mb{R}^n}u_t(t,x)\,\mathrm{d}x-\beta \varepsilon\int_{\mb{R}^n} u_2(x)\,\mathrm{d}x-\varepsilon\int_{\mb{R}^n} u_1(x)\, \mathrm{d}x  =\int_0^t\int_{\mb{R}^n}|u(s,x)|^p\,\mathrm{d}x\,\mathrm{d}s.
\end{align*} 
Differentiating the previous relation with respect to $t$, we get
\begin{align} \label{ODE for F}
\beta U'''(t)+U''(t) = \int_{\mathbb{R}^n} |u(t,x)|^p \, \mathrm{d} x.
\end{align}
By employing H\"older's inequality and \eqref{property of finite speed of propagation u(t,.)}, we may estimate
\begin{align*}
\int_{\mathbb{R}^n} |u(t,x)|^p \, \mathrm{d} x \geqslant C (R+t)^{-n(p-1)} |U(t)|^p,
\end{align*} where $C=C(n,p)>0$ is a constant that depends on the measure of the unitary ball. Hence, from \eqref{ODE for F} we obtain 
\begin{align} \label{ODI for F}
\beta U'''(t)+U''(t)  \geqslant C (R+t)^{-n(p-1)} |U(t)|^p.
\end{align} The previous ordinary differential inequality for $U$ allows us to the derive the frame for our iteration argument. In other words, integrating twice, by \eqref{ODI for F} we have
\begin{align*}
\beta U'(s)+U(s)  \geqslant \beta U'(0)+U(0)+ \left(\beta U''(0)+U'(0)\right) s+   C \int_0^s \int_0^\sigma (R+\tau)^{-n(p-1)} |U(\tau)|^p \, \mathrm{d}\tau \, \mathrm{d}\sigma.
\end{align*} Then, multiplying the last inequality by $\mathrm{e}^{s/\beta}$ and  integrating over $[0,t]$, we arrive at
\begin{align}
U(t)  & \geqslant  U(0)\, \mathrm{e}^{-t/\beta}+\left( \beta  U'(0)+  U(0)\right) \left(1-\mathrm{e}^{-t/ \beta }\right)+ \left( U''(0)+ \tfrac{1}{\beta}U'(0)\right)\left(\beta (t-\beta)+\beta^2 \mathrm{e}^{-t/\beta}\right) \notag\\ & \quad +   \tfrac{C}{\beta} \int_0^t \mathrm{e}^{(s-t)/ \beta} \int_0^s \int_0^\sigma (R+\tau)^{-n(p-1)} |U(\tau)|^p \, \mathrm{d}\tau \, \mathrm{d}\sigma \, \mathrm{d}s. \label{lower bound F complete}
\end{align}
From the integral inequality \eqref{lower bound F complete} we have a twofold consequence. Since we assume that initial data are nonnegative, then \eqref{lower bound F complete} implies $U(t) \gtrsim \varepsilon $ for any $t\geqslant 0$. So, in particular, $U$ is a positive function.


On the other hand, if we neglect the terms involving $U(0),U'(0)$ and $U''(0)$ in \eqref{lower bound F complete}, then, we find 
\begin{align}
U(t)  & \geqslant   \tfrac{C}{\beta} \int_0^t \mathrm{e}^{(s-t)/ \beta} \int_0^s \int_0^\sigma (R+\tau)^{-n(p-1)} (U(\tau))^p \, \mathrm{d}\tau \, \mathrm{d}\sigma \, \mathrm{d}s. \label{iteration frame F}
\end{align} We point out explicitly that \eqref{iteration frame F} will play a fundamental role in our iteration argument: this is, in fact, the frame which allows us to determine a sequence of lower bound estimates for the function $U$.

\subsection{Lower bound for the functional}
Even though we proved that $U(t)\gtrsim \varepsilon$ in the last subsection, this lower bound for $U$ is too weak in order to start with the iteration procedure. For this reason we will improve this lower bound for $U$ by introducing a second time -- dependent functional. Let us consider the function
\begin{equation} \label{def eigenfunction laplacian}
\begin{split}
\Phi(x) & \doteq  \mathrm{e}^{x}+\mathrm{e}^{-x} \quad \ \, \qquad \ \mbox{if} \ n=1, \\  \Phi(x) & \doteq \int_{\mathbb{S}^{n-1}} \mathrm{e}^{x\cdot \omega} \, \mathrm{d} \sigma_\omega \qquad \mbox{if} \ n\geqslant 2. 
\end{split}
\end{equation} This function has been introduced for the first time in the study of blow -- up results for wave models in \cite{YordanovZhang2006}. The function $\Phi$ is a positive smooth function that satisfies the following crucial properties:
\begin{align}
& \Delta \Phi =\Phi , \label{Laplace Phi =Phi} \\
& \Phi (x) \sim |x|^{-\frac{n-1}{2}}  \mathrm{e}^x \qquad \mbox{as} \ |x|\to \infty. \label{Asymptotic behavior Phi}
\end{align}
Furthermore, we introduce the function with separate variables $\Psi=\Psi(t,x)=\mathrm{e}^{-t}\Phi(x)$. Clearly, $\Psi$ is a solution of the adjoint equation to the homogeneous linear MGT equation, namely, 
\begin{align} \label{adjoint MGT eq}
-\beta\, \partial_t^3 \Psi+\partial_t^2 \Psi -\Delta \Psi +\beta \Delta\partial_t \Psi=0.
\end{align} 

We can introduce now the definition of the second functional $U_1$ as follows:
\begin{align*}
U_1(t) \doteq \int_{\mathbb{R}^n} u(t,x) \Psi(t,x) \, \mathrm{d}x.
\end{align*} Since $\Psi$ is a positive function, applying \eqref{Def.Energy solution} with test function $\Psi$, we get
\begin{align*}
0 & \leqslant \int_0^t \int_{\mathbb{R}^n} |u(s,x)|^p \Psi(s,x) \, \mathrm{d}x \, \mathrm{d}s \\
& =\int_0^t \int_{\mathbb{R}^n} \Big(-\beta\,  \Psi_{ttt}(s,x)+ \Psi_{tt}(s,x) -\Delta \Psi(s,x) +\beta \Delta \Psi_t(s,x)\Big) u(s,x) \, \mathrm{d}x \, \mathrm{d}s \\ 
& \qquad + \beta \int_{\mathbb{R}^n} \big(u_{tt}(s,x)\Psi(s,x)-u_t(s,x)\Psi_t(s,x)+u(s,x)\Psi_{tt}(s,x)\big) \mathrm{d}x \, \Big|_{s=0}^{s=t} \\ 
& \qquad +  \int_{\mathbb{R}^n} \big(u_t(s,x)\Psi(s,x)-u(s,x)\Psi_t(s,x)\big) \mathrm{d}x \, \Big|_{s=0}^{s=t}- \beta \int_{\mathbb{R}^n}u(s,x) \Delta\Psi(s,x) \, \mathrm{d}x   \, \Big|_{s=0}^{s=t}\, \\
& =   \int_{\mathbb{R}^n} \big(\beta u_{tt}(s,x)\Psi(s,x)+ (\beta+1)u_t(s,x)\Psi(s,x)+u(s,x)\Psi(s,x)\big) \mathrm{d}x \, \Big|_{s=0}^{s=t} \\
&= \beta U_1''(t)+(3\beta+1)U_1'(t)+(2\beta+2)U_1(t) -\left(\beta U_1''(0)+(3\beta+1)U_1'(0)+(2\beta+2)U_1(0)\right),
\end{align*} where in the second last step we used that $\Psi $ solves \eqref{adjoint MGT eq}, while in the last step we used the obvious representations
\begin{align*}
U_1'(t) &= \int_{\mathbb{R}^n} \big(u_t(s,x)\Psi(s,x)+u(s,x)\Psi_t(s,x)\big) \mathrm{d} x,\\
U_1''(t) &= \int_{\mathbb{R}^n} \big(u_{tt}(s,x)\Psi(s,x)+2u_t(s,x)\Psi_t(s,x)+u(s,x)\Psi_{tt}(s,x)\big) \mathrm{d} x.
\end{align*}
Note that we may employ $\Psi$ as test function even though it has no compact support thanks to the support property for $u$ in \eqref{support condition u}.
As outcome of the previous chain of equalities, we get that
\begin{align}
U''_1(t)+\left(3+\tfrac{1}{\beta}\right) U'_1(t)+\left(2+\tfrac{2}{\beta}\right) U_1(t) \geqslant  U''_1(0)+\left(3+\tfrac{1}{\beta}\right) U'_1(0)+\left(2+\tfrac{2}{\beta}\right) U_1(0). \label{ODI F1}
\end{align} We can rewrite the left -- hand side of \eqref{ODI F1} as 
\begin{align*}
\mathrm{e}^{-(1+1/\beta)t} \frac{\mathrm{d}}{\mathrm{d} t} \left\{ \mathrm{e}^{(1+1/\beta)t} \left[ \mathrm{e}^{-2t}  \frac{\mathrm{d}}{\mathrm{d} t}\left( \mathrm{e}^{2t} U_1(t)\right)\right]\right\},
\end{align*} while the right -- hand side depends only on initial data
\begin{align*}
U''_1(0)+\left(3+\tfrac{1}{\beta}\right) U'_1(0)+\left(2+\tfrac{2}{\beta}\right) U_1(0)& = \varepsilon \int_{\mathbb{R}^n} \left( u_2(x)+\tfrac{\beta+1}{\beta}u_1(x)+\tfrac{1}{\beta}u_0(x)\right)\Phi(x)\, \mathrm{d} x \\ & = \varepsilon  I_\beta [u_0,u_1,u_2]  >0,
\end{align*} where 
\begin{align*}
I_\beta [u_0,u_1,u_2] \doteq  \int_{\mathbb{R}^n} \big(u_2(x)+\tfrac{\beta+1}{\beta}u_1(x)+\tfrac{1}{\beta}u_0(x)\big)\Phi(x)\, \mathrm{d} x.
\end{align*}


Multiplying \eqref{ODI F1} by $ \mathrm{e}^{(1+1/\beta)t}$ and integrating over $[0,t]$, we get
\begin{align*}
\mathrm{e}^{-2t}  \frac{\mathrm{d}}{\mathrm{d} t}\left( \mathrm{e}^{2t} U_1(t)\right) \geqslant \left(U'_1(0)+2U_1(0)\right)\mathrm{e}^{-(1+1/\beta)t}+ \varepsilon \tfrac{\beta}{\beta+1}  I_\beta [u_0,u_1,u_2] \left(1- \mathrm{e}^{-(1+1/\beta)t}\right).
\end{align*}
Analogously to the last step, we multiply the previous inequality by  $\mathrm{e}^{2t}$ and we integrate over $[0,t]$, so that 
\begin{align*}
U_1(t) & \geqslant  U_1(0) \, \mathrm{e}^{-2t}+\tfrac{\beta}{\beta-1} \left(U'_1(0)+2U_1(0)\right) \left(\mathrm{e}^{-(1+1/\beta)t}-\mathrm{e}^{-2t}\right)\\ & \quad+\tfrac{\beta}{\beta+1}\left(U''_1(0)+\left(3+\tfrac{1}{\beta}\right) U'_1(0)+\left(2+\tfrac{2}{\beta}\right) U_1(0) \right) \left(\tfrac{1}{2}\left(1-\mathrm{e}^{-2t}\right)-\tfrac{\beta}{\beta-1}\left(\mathrm{e}^{-(1+1/\beta)t}-\mathrm{e}^{-2t}\right)\right).
\end{align*} for $\beta\neq1$, while for $\beta = 1$ we get
\begin{align*}
U_1(t)\geqslant U_1(0) \, \mathrm{e}^{-2t}+(U_1'(0)+2U_1(0)) \, t \mathrm{e}^{-2t}+\tfrac{1}{2}\left(\tfrac{1}{2}\left(1-\mathrm{e}^{-2t}\right)-t\mathrm{e}^{-2t}\right)\left(U''_1(0)+4 U'_1(0)+4 U_1(0)\right).
\end{align*}
Therefore, thanks to the assumptions on $u_0,u_1,u_2$, the previous estimates yield easily
\begin{align} \label{lower bound F1}
U_1(t)\gtrsim \varepsilon,
\end{align} where the unexpressed multiplicative constant depends on $u_0,u_1,u_2$.

Let us show now how \eqref{lower bound F1} provides a lower bound estimate for the spatial integral of the nonlinearity $|u|^p$. 
Applying H\"older's inequality, we have
\begin{align*}
\varepsilon \lesssim U_1(t) \leqslant \left(\ \int_{\mathbb{R}^n}|u(t,x)|^p \mathrm{d} x\right)^{\frac{1}{p}} \left(\ \int_{B_{R+t}}\Psi(t,x)^{p'}  \mathrm{d} x\right)^{\frac{1}{p'}},
\end{align*} where $p'$ denotes the conjugate exponent of $p$. In the literature, it is well -- known  that the $p'$ power of the  $L^{p'}(B_{R+T}) $ -- norm of $\Psi(t,\cdot)$ can be estimate in the following way:
\begin{align*}
\int_{B_{R+t}}\Psi(t,x)^{p'}  \mathrm{d} x  \lesssim (R+t)^{n-1-\frac{n-1}{2}p'}
\end{align*} (cf. \cite[Estimate (2.5)]{YordanovZhang2006}), thus, we find 
\begin{align}\label{lower bound int |u|^p}
\int_{\mathbb{R}^n}|u(t,x)|^p \mathrm{d} x \geqslant K \varepsilon^p (R+t)^{n-1-\frac{n-1}{2}p}
\end{align} for a suitable positive constant $K$.

Finally, we combine \eqref{ODE for F} and \eqref{lower bound int |u|^p} in order to get a lower bound for $U$ which will allow us to start with the iteration argument. Repeating the same intermediate steps that we did in order to prove \eqref{iteration frame F} starting from \eqref{ODI for F}, from \eqref{ODE for F} we find
\begin{align*}
U(t)  & \geqslant   \tfrac{1}{\beta} \int_0^t \mathrm{e}^{(s-t)/ \beta} \int_0^s \int_0^\sigma \int_{\mathbb{R}^n} |u(\tau,x)|^p \, \mathrm{d}x \, \mathrm{d}\tau \, \mathrm{d}\sigma \, \mathrm{d}s.
\end{align*} Next, we plug the lower bound \eqref{lower bound int |u|^p} in the last estimate. Thus, we obtain
\begin{align*}
U(t) & \geqslant   \tfrac{K}{\beta} \,  \varepsilon^p \int_0^t \mathrm{e}^{(s-t)/ \beta} \int_0^s \int_0^\sigma   (R+\tau)^{n-1-\frac{n-1}{2}p}  \, \mathrm{d}\tau \, \mathrm{d}\sigma \, \mathrm{d}s \\
& \geqslant   \tfrac{K}{\beta} \,  \varepsilon^p  (R+t)^{-\frac{n-1}{2}p} \int_0^t \mathrm{e}^{(s-t)/ \beta} \int_0^s \int_0^\sigma   \tau^{n-1}  \, \mathrm{d}\tau \, \mathrm{d}\sigma \, \mathrm{d}s\\
& =   \tfrac{K}{ n(n+1)}  \tfrac{1}{\beta}\,  \varepsilon^p  (R+t)^{-\frac{n-1}{2}p} \int_0^t \mathrm{e}^{(s-t)/ \beta}  s^{n+1} \, \mathrm{d}s \\
& \geqslant   \tfrac{K}{ n(n+1)}  \tfrac{1}{\beta}\,  \varepsilon^p  (R+t)^{-\frac{n-1}{2}p} \int_{t/2}^t \mathrm{e}^{(s-t)/ \beta}  s^{n+1} \, \mathrm{d}s \\ & \geqslant   \tfrac{K}{2^{n+1} n(n+1)} \,  \varepsilon^p  (R+t)^{-\frac{n-1}{2}p} \,  t^{n+1} \int_{t/2}^t  \tfrac{1}{\beta} \, \mathrm{e}^{(s-t)/ \beta}  \, \mathrm{d}s 
\\& =  \tfrac{K}{2^{n+1} n(n+1)} \,  \varepsilon^p  (R+t)^{-\frac{n-1}{2}p} \,  t^{n+1} \left(1-\mathrm{e}^{-t/(2\beta)}\right).
\end{align*} In particular, for $t\geqslant \beta$ the factor containing the exponential function in the last line of the previous chain of inequalities can be estimate from below by a constant, namely,
\begin{align} \label{lower bound F 0 step}
U(t) \geqslant C_0 (R+t)^{-\alpha_0} \, t^{\gamma_0} \qquad \mbox{for any} \ t\geqslant \beta,
\end{align} where the multiplicative constant is 
\begin{align*}
C_0\doteq K 2^{-(n+1)}( n(n+1))^{-1} (1-\mathrm{e}^{-1/2}) \,  \varepsilon^p
\end{align*} 
and the exponents are defined by
\begin{align*}
\alpha_0\doteq \frac{n-1}{2}p\quad\mbox{and}\quad \gamma_0\doteq n+1.
\end{align*}

%

\subsection{Iteration argument} \label{Subsection Iter Arg}
In the previous subsection, we derived a first lower bound for $U$. Now we will derive a sequence of lower bounds for $U$ by using the iteration frame \eqref{iteration frame F}. More precisely, we will show that
\begin{align}\label{sequence of lower bound F}
U(t) \geqslant C_j (R+t)^{-\alpha_j} (t- L_j \beta)^{\gamma_j} \qquad \mbox{for any} \ t\geqslant L_j \beta,
\end{align} where $\{C_j\}_{j\in \mathbb{N}}$, $\{\alpha_j\}_{j\in \mathbb{N}}$  and $\{\gamma_j\}_{j\in \mathbb{N}}$ are sequences of nonnegative real numbers that we will determine throughout the proof and $\{L_j\}_{j\in\mathbb{N}}$ is the sequence of the partial products of the convergent infinite product
\begin{align*}
\prod_{k=0}^\infty \ell_k \quad  \mbox{with} \ \ \ell_k\doteq 1+p^{-k} \ \ \mbox{for any} \ k\in \mathbb{N},
\end{align*} that is, $$L_j\doteq \prod_{k=0}^{j} \ell_k \quad \mbox{for any} \ j\in \mathbb{N}.$$ 

Note that \eqref{lower bound F 0 step} implies \eqref{sequence of lower bound F} for $j=0$. We are going to prove \eqref{sequence of lower bound F} by using an inductive argument. Therefore, it remains to prove just the inductive step. Let us assume the validity of \eqref{sequence of lower bound F} for $j\geqslant 0$. We will prove \eqref{sequence of lower bound F} for $j+1$. After shrinking the domain of integration in \eqref{iteration frame F}, if we plug \eqref{sequence of lower bound F} in \eqref{iteration frame F}, we get
\begin{align*}
U(t)  & \geqslant    \tfrac{C}{\beta} \int_{L_j \beta}^t \mathrm{e}^{(s-t)/ \beta} \int_{L_j \beta}^s \int_{L_j \beta}^\sigma (R+\tau)^{-n(p-1)} (U(\tau))^p \, \mathrm{d}\tau \, \mathrm{d}\sigma \, \mathrm{d}s \\
& \geqslant    \tfrac{C}{\beta} \,  C_j^p  \int_{L_j \beta}^t \mathrm{e}^{(s-t)/ \beta} \int_{L_j \beta}^s \int_{L_j \beta}^\sigma (R+\tau)^{-n(p-1)-\alpha_j p} (\tau- L_j \beta)^{\gamma_j p}\, \mathrm{d}\tau \, \mathrm{d}\sigma \, \mathrm{d}s \\
& \geqslant    \tfrac{C}{\beta} \,  C_j^p  (R+t)^{-n(p-1)-\alpha_j p} \int_{L_j \beta}^t \mathrm{e}^{(s-t)/ \beta} \int_{L_j \beta}^s \int_{L_j \beta}^\sigma  (\tau- L_j \beta)^{\gamma_j p}\, \mathrm{d}\tau \, \mathrm{d}\sigma \, \mathrm{d}s \\
& \geqslant    \tfrac{C}{\beta} (\gamma_j p+1)^{-1}(\gamma_j p+2)^{-1} \,  C_j^p  (R+t)^{-n(p-1)-\alpha_j p} \int_{L_j \beta}^t \mathrm{e}^{(s-t)/ \beta} (s- L_j \beta)^{\gamma_j p+2}\,  \mathrm{d}s \\
& \geqslant    \tfrac{C}{\beta} (\gamma_j p+1)^{-1}(\gamma_j p+2)^{-1} \,  C_j^p  (R+t)^{-n(p-1)-\alpha_j p} \int_{t/\ell_{j+1}}^t \mathrm{e}^{(s-t)/ \beta} (s- L_j \beta)^{\gamma_j p+2}\,  \mathrm{d}s
\end{align*} for $t\geqslant L_{j+1} \beta$. Note that in the last step we could restrict the domain of integration with respect to $s$ from $[L_j\beta,t]$ to $[t/\ell_{j+1},t]$ because $t\geqslant L_{j+1}\beta$ and $\ell_{j+1}>1$ imply $L_{j} \beta \leqslant t/\ell_{j+1}<t$. Consequently,
\begin{align*}
U(t)  & \geqslant  \frac{C C_j^p}{(\gamma_j p+1)(\gamma_j p+2) \, \ell_{j+1}^{\gamma_jp+2}} \,   (R+t)^{-n(p-1)-\alpha_j p} (t- L_j \ell_{j+1} \beta)^{\gamma_j p+2} \int_{t/\ell_{j+1}}^t \tfrac{1}{\beta}\,\mathrm{e}^{(s-t)/ \beta} \, \mathrm{d}s \\
& = \frac{C C_j^p}{(\gamma_j p+1)(\gamma_j p+2) \,\ell_{j+1}^{\gamma_jp+2}} \,   (R+t)^{-n(p-1)-\alpha_j p} (t- L_{j+1} \beta)^{\gamma_j p+2} \left(1-\mathrm{e}^{ -(t/ \beta)(1-1/\ell_{j+1})} \right)
\end{align*} for $t\geqslant L_{j+1} \beta$. Finally, we remark that for $t\geqslant L_{j+1} \beta \geqslant \ell_{j+1} \beta $ we may estimate 
\begin{align}
1-\mathrm{e}^{ -(t/ \beta)(1-1/\ell_{j+1})} & \geqslant 1-\mathrm{e}^{-(\ell_{j+1}-1)} \geqslant 1- \left(1-(\ell_{j+1}-1)+\tfrac12 (\ell_{j+1}-1)^2\right) \notag \\
& = (\ell_{j+1}-1)\left(1-\tfrac12 (\ell_{j+1}-1)\right) = p^{-(j+1)}\left(1-1/(2p^{j+1})\right) \notag \\
& =  p^{-2(j+1)}\left(p^{j+1}-1/2\right)\geqslant \left(p-1/2\right)  p^{-2(j+1)}.
\end{align}
Also, for $t\geqslant L_{j+1} \beta$ we have shown
\begin{align*}
U(t)  & \geqslant  \frac{(p-1/2) C C_j^p p^{-2(j+1)}}{(\gamma_j p+1)(\gamma_j p+2) \,\ell_{j+1}^{\gamma_jp+2}} \,   (R+t)^{-n(p-1)-\alpha_j p} (t- L_{j+1} \beta)^{\gamma_j p+2},
\end{align*}
which is exactly \eqref{sequence of lower bound F} for $j+1$, provided that
\begin{align*}
C_{j+1} & \doteq   \frac{(p-1/2) C C_j^p p^{-2(j+1)}}{(\gamma_j p+1)(\gamma_j p+2) \,\ell_{j+1}^{\gamma_jp+2}}, \ \
\alpha_{j+1}  \doteq n(p-1)+p \alpha_j  , \ \
\gamma_{j+1}  \doteq  2+p \gamma_j .
\end{align*}

\subsection{Upper bound estimate for the lifespan}

In the last subsection, we determined the sequence of lower bound estimates in \eqref{sequence of lower bound F} for $U$. Now we want to show that the $j$ -- dependent lower bound in \eqref{sequence of lower bound F} for $U$ blows up in finite time as $j\to \infty$. This will provide the desired blow -- up result and an upper bound estimate for the lifespan as well. Let us get started by estimating the multiplicative constant $C_j$ in a suitable way.

In order to estimate $C_j$ from below we have to determine first the explicit representation for $\gamma_{j}$. Since $\alpha_j = n(p-1)+p \alpha_{j-1}  $ and $\gamma_j= 2+p\gamma_{j-1}$, applying recursively these relations, we get
\begin{align}
\alpha_j & = p^{2}\alpha_{j-2}+n(p-1) (1+p) = \cdots =   p^{j}\alpha_0+n(p-1) (1+p+\cdots +p^{j-1})  = (\alpha_0+n)p^j -n , \label{representation alpha j} \\
\gamma_j &= p^2\gamma_{j-2}+2(1+p) = \cdots =   p^{j}\gamma_0+2 (1+p+\cdots +p^{j-1}) =\left(\gamma_0+\tfrac{2}{p-1}\right) p^{j}-\tfrac{2}{p-1}. \label{representation gamma j}
\end{align} Therefore, from \eqref{representation gamma j} we have
\begin{align*}
(\gamma_{j-1}p+1)(\gamma_{j-1}p+2)\leqslant (\gamma_{j-1}p+2)^2 = \gamma_j^2 \leqslant \left(\gamma_0+\tfrac{2}{p-1}\right)^2 p^{2j}
\end{align*} which implies in turns
\begin{align*}
C_j \geqslant  (p-1/2) C \left(\gamma_0+\tfrac{2}{p-1}\right)^{- 2}C_{j-1}^p p^{-4j} \ell_{j}^{-\gamma_{j}}.
\end{align*}
Moreover, it holds
\begin{align*}
\lim_{j\to \infty} \ell_j^{\gamma_j}= \lim_{j\to \infty} \exp\left(\left(\gamma_0+\tfrac{2}{p-1}\right)p^{j}\log \left(1+p^{-j}\right)  \right) = \mathrm{e}^{\gamma_0+2/(p-1)},
\end{align*} so, in particular, we can find a suitable constant $M=M(n,p)>0$ such that $\ell_j^{-\gamma_j}\geqslant M$ for any $j\in \mathbb{N}$. Hence, 
\begin{align*}
C_j \geqslant \underbrace{ (p-1/2) C M \left(\gamma_0+\tfrac{2}{p-1}\right)^{- 2}}_{\doteq D}C_{j-1}^p p^{-4j} \quad \mbox{for any} \ j\in\mathbb{N}.
\end{align*}
Applying the logarithmic function to both sides of the inequality $C_j\geqslant D C_{j-1}^p p^{-4j}$ and using iteratively the resulting inequality, we get
\begin{align*}
\log C_j & \geqslant p \log C_{j-1} -4 j \log p +\log D \\
& \geqslant p^2 \log C_{j-2} -4( j+(j-1)p ) \log p + (1+p)\log D \\
& \geqslant \cdots \geqslant p^{j}\log C_0 -4 \left(\sum_{k=0}^{j-1} (j-k)p^k\right) \log p+ \left(\sum_{k=0}^{j-1} p^k\right)\log D.
\end{align*} Using the identities
\begin{align} \label{identity sum (j-k)p^k}
\sum_{k=0}^{j-1} (j-k)p^k = \frac{1}{p-1} \left(\frac{p^{j+1}-p}{p-1}-j\right)\quad\text{and}\quad  \sum_{k=0}^{j-1} p^k = \frac{p^j-1}{p-1},
\end{align} it follows
\begin{align*}
\log C_j \geqslant p^j \left(\log C_0-\frac{4p \log p}{(p-1)^2} +\frac{\log D}{p-1} \right)+\frac{4j \log p}{p-1}+\frac{4p \log p}{(p-1)^2}-\frac{\log D}{p-1}
\end{align*} for any $j\in\mathbb{N}$. Let $j_0=j_0(n,p)\in\mathbb{N}$ be the smallest nonnegative integer such that
\begin{align*}
j_0\geqslant \frac{\log D}{4\log p}-\frac{p}{p-1}.
\end{align*}
Then, for any $j\geqslant j_0$ it results
\begin{align}\label{lower bound log Cj}
\log C_j \geqslant p^j \left(\log C_0-\frac{4p \log p}{(p-1)^2} +\frac{\log D}{p-1} \right) = p^j \log \left(D^{1/(p-1)}p^{-(4p)/(p-1)^2}C_0\right) =p^j \log (E_0 \varepsilon^p)
\end{align} for a suitable constant $E_0=E_0(n,p)>0$.

Let us denote $$L\doteq \lim_{j\to\infty} L_j = \prod_{j=0}^\infty \ell_{j}\in \mathbb{R}.$$ Note that thanks to $\ell_j>1$, it holds $L_j \uparrow L$ as $j\to \infty$. So, in particular, \eqref{sequence of lower bound F} holds for any $j\in\mathbb{N}$ and any $t\geqslant L\beta$.

Combining \eqref{sequence of lower bound F}, \eqref{representation alpha j}, \eqref{representation gamma j} and \eqref{lower bound log Cj}, we find
\begin{align*}
U(t) & \geqslant \exp \left(p^j \log (E_0\varepsilon^p)\right) (R+t)^{-\alpha_j} (t- L\beta)^{\gamma_j} \\
& = \exp \left(p^j \! \left( \log (E_0\varepsilon^p)-(\alpha_0+n)\log(R+t)+\left(\gamma_0+\tfrac{2}{p-1}\right)\log(t-L\beta)\right)\!\right) \! 
(R+t)^{n} (t- L\beta)^{-2/(p-1)}
\end{align*}
for any $j\geqslant j_0$ and any $t\geqslant L\beta$. Finally, for $t\geqslant \max\{R,2L\beta\}$, since $R+t\leqslant 2t$ and $t-L\beta\geqslant t/2$, we have
\begin{align} \label{final lower bound F}
U(t) & \geqslant  \exp \left(p^j \log \left(E_1\varepsilon^p t^{\gamma_0+\frac{2}{p-1}-(\alpha_0+n)}\right)\right) (R+t)^{n} (t- L\beta)^{-2/(p-1)}
\end{align}
for any $j\geqslant j_0$, where $E_1\doteq2^{-\left(\alpha_0+n+\gamma_0+2/(p-1)\right)}E_0$. We can rewrite the exponent for $t$ in the last inequality as follows:
\begin{align*}
\gamma_0+\tfrac{2}{p-1}-(\alpha_0+n) & = -\tfrac{n-1}{2}p+1+\tfrac{2}{p-1}= \tfrac{1}{2(p-1)}\left(2+(n+1)p-(n-1)p^2\right)= \tfrac{\theta(p,n)}{2(p-1)},
\end{align*} where $\theta(p,n)$ is defined in \eqref{def theta(p,n)}. Therefore, for $1<p<p_{\mathrm{Str}}(n)$ (respectively, for $1<p$ when $n=1$), the exponent for $t$ is positive. Let us fix $\varepsilon_0=\varepsilon_0(u_0,u_1,u_2,n,p,R,\beta)>0$ such that
\begin{align*}
\varepsilon_0^{-\frac{2p(p-1)}{\theta(p,n)}}\geqslant E_1^{\frac{2(p-1)}{\theta(p,n)}} \max\{R,2L\beta\}.
\end{align*} Consequently, for any $\varepsilon\in(0,\varepsilon_0]$ and any $t>E_2  \varepsilon^{-\frac{2p(p-1)}{\theta(p,n)}}$, where $E_2\doteq E_1^{-\frac{2(p-1)}{\theta(p,n)}}$, we obtain
\begin{align*}
t\geqslant \max\{R,2L\beta\} \quad \mbox{and} \quad \log \bigg(E_1\varepsilon^p t^{ \tfrac{\theta(p,n)}{2(p-1)}}\bigg)>0.
\end{align*} Also, for any $\varepsilon\in(0,\varepsilon_0]$ and any $t>E_2  \varepsilon^{-\frac{2p(p-1)}{\theta(p,n)}}$ letting $j\to \infty$ in \eqref{final lower bound F} we see that the lower bound for $U(t)$ blows up. Thus, for any $\varepsilon\in(0,\varepsilon_0]$ the functional $U$ has to blow up in finite time and, moreover, the lifespan of the local solution $u$ can be estimated from above as follows: $$T(\varepsilon) \lesssim \varepsilon^{-\frac{2p(p-1)}{\theta(p,n)}}.$$
In conclusion, the proof of Theorem \ref{Thm blow-up sub Strauss} is complete.

\section{Blow -- result in the critical case} \label{Section blow up critical case} 

In this section, we shall prove a blow-up result for the semilinear MGT equation in the conservative case with power nonlinearity in the critical case, that is, we are interested in the Cauchy problem \eqref{Semi.MGT.Equation.epsilon}
in the case in which the exponent of the nonlinear term is
$p=p_{\mathrm{Str}}(n)$ (clearly, provided that $n\geqslant 2$). 

Our approach is based on the technique developed in \cite{WakYor18}, where the slicing procedure is applied in order to get a blow -- up result for the semilinear wave  equation in the critical case.  Nonetheless, the parameters which characterize the slicing procedure itself are chosen in a more suitable way for the MGT equation (cf. Section \ref{Subsection iteration method crit case}).

\subsection{Auxiliary functions}

Let us recall the definition of a pair of auxiliary functions from \cite{WakYor18}, which are necessary in order to introduce the time -- dependent functional that will be considered for the iteration argument in the critical case $p=p_{\mathrm{Str}}(n)$. 

Let $r>-1$ be a real parameter.
We consider the function $\Phi$ defined by \eqref{def eigenfunction laplacian}.
Then, we introduce the couple of auxiliary functions
\begin{align}
\xi_r(t,x) & \doteq  \int_0^{\lambda_0} \mathrm{e}^{-\lambda(t+R)} \cosh (\lambda t) \, \Phi(\lambda x) \, \lambda^r \, \mathrm{d}\lambda, \label{def xi}\\
\eta_r(t,s,x) & \doteq  \int_0^{\lambda_0} \mathrm{e}^{-\lambda(t+R)} \frac{\sinh (\lambda (t-s))}{\lambda(t-s)} \,\Phi(\lambda x) \,\lambda^r \, \mathrm{d}\lambda, \label{def eta}
\end{align} where $\lambda_0$ is a fixed positive parameter.

Some useful properties of $\xi_r$ and $\eta_r$ are stated in the following lemma, whose proof can be found in \cite[Lemma 3.1]{WakYor18}.

\begin{lem} \label{lemma eta and xi estimates}Let $n\geqslant 2$. There exists $\lambda_0>0$ such that the following properties hold:
	\begin{itemize}
		\item[\rm{(i)}] if $r>-1$, $|x|\leqslant R$ and $t\geqslant 0$, then, 
		\begin{align*}
		\xi_r(t,x) & \geqslant A_0, \\
		\eta_r(t,0,x) & \geqslant B_0 \langle t\rangle^{-1};
		\end{align*}
		\item[\rm{(ii)}] if $r>-1$, $|x|\leqslant s+R$ and $t>s\geqslant 0$, then, 
		\begin{align*}
		\eta_r(t,s,x) & \geqslant B_1 \langle t\rangle^{-1} \langle s\rangle^{-r};
		\end{align*}
		\item[\rm{(iii)}] if $r>(n-3)/2$, $|x|\leqslant t+R$ and $t> 0$, then, 
		\begin{align*}
		\eta_r(t,t,x) & \leqslant B_2 \langle t\rangle^{-\frac{n-1}{2}} \langle t-|x| \rangle^{\frac{n-3}{2}-r}.
		\end{align*}
	\end{itemize}
	Here $A_0$ and $B_k$, with $k=0,1,2$, are positive constants depending only on $\lambda_0$, $r$ and $R$ and we denote $\langle y\rangle \doteq 3+|y|$.
\end{lem}

\begin{rem} Even though in \cite{WakYor18} the previous lemma is stated by assuming $r>0$ in {\rm(i)} and {\rm(ii)}, the proof provided in that paper holds true for any $r>-1$ as well.
\end{rem}

\subsection{Main result}

Throughout Section \ref{Section blow up critical case} we will consider a slightly different notion of energy solutions to \eqref{Semi.MGT.Equation.epsilon} with respect to the one given in Section \ref{Section blow up} (cf. Definition \ref{Definition Energy solution}).  
Before introducing this different notion of energy solutions to \eqref{Semi.MGT.Equation.epsilon}, let us recall that $u$ solves the Cauchy problem \eqref{Semi.MGT.Equation.epsilon} if and only if it solves the second order Cauchy problem
\begin{align} \label{Semi MGT |u|^p memory term} 
\begin{cases} 
u_{tt} - \Delta u= \varepsilon \, \mathrm{e}^{-t/\beta} \left(u_2(x)-\Delta u_0(x)\right)+ \frac{1}{\beta} \displaystyle{\int_0^t  \mathrm{e}^{-(t-s)/\beta}|u(s,x)|^p \, \mathrm{d}s},  & x\in \mathbb{R}^n, \ t>0, \\
(u,u_t)(0,x)= \varepsilon (u_0,u_1)(x), & x\in \mathbb{R}^n.
\end{cases}
\end{align} 

Having this fact in mind, it is quite natural to introduce also the following reasonable notion of energy solutions to \eqref{Semi MGT |u|^p memory term} and, then, to \eqref{Semi.MGT.Equation.epsilon}.

\begin{defn}\label{Defn.Energy.Solution}
	Let $(u_0,u_1,u_2)\in H^2(\mathbb{R}^n)\times H^1(\mathbb{R}^n)\times L^2(\mathbb{R}^n)$. We say that 
	\begin{align*}
	u\in\mathcal{C}\big([0,T),H^2(\mathbb{R}^n)\big)\cap \mathcal{C}^1\big([0,T),H^1(\mathbb{R}^n)\big) \cap \mathcal{C}^2\big([0,T),L^2(\mathbb{R}^n)\big){\cap L^p_{\mathrm{loc}}([0,T)\times\mb{R}^n)}
	\end{align*}  is an energy solution of \eqref{Semi.MGT.Equation.epsilon} on $[0,T)$ if $u$
	fulfills $u(0,\cdot)=\varepsilon u_0$ in $H^2(\mathbb{R}^n)$ and the  integral relation
	\begin{align}\label{Eq.Defn.Energy.Solution}
	\int_{\mathbb{R}^n}& u_t(t,x)\, \psi(t,x)\,\mathrm{d}x-\varepsilon\int_{\mathbb{R}^n}u_1(x)\, \psi(0,x)\,\mathrm{d}x\notag\\ &+\int_0^t\int_{\mathbb{R}^n}\left(\nabla_x u(s,x)\cdot\nabla_x\psi(s,x)- u_t(s,x)\, \psi_s(s,x)\right)\mathrm{d}x\,\mathrm{d}s\notag\\
	& =  \varepsilon \int_0^t \mathrm{e}^{-s/\beta } \int_{\mathbb{R}^n}\psi(s,x) \left( u_2(x)-\Delta u_0(x)\right)\mathrm{d}x\,\mathrm{d}s\notag\\
	&\quad + \frac1\beta \int_0^t \int_0^s \mathrm{e}^{ (\tau-s)/\beta} \int_{\mathbb{R}^n}  \psi(s,x) |u(\tau,x)|^p\,\mathrm{d}x \,\mathrm{d}\tau \, \mathrm{d}s
	\end{align}
	for any $\psi\in\mathcal{C}_0^{\infty}([0,T)\times\mathbb{R}^n)$ and any $t\in[0,T)$.
\end{defn}

Note that, performing a further step of integration by parts in \eqref{Eq.Defn.Energy.Solution}, it results
\begin{align}
\int_{\mathbb{R}^n}&\left(\psi(t,x)\, u_t(t,x)-\psi_s(t,x)\, u(t,x)\right)\mathrm{d}x -\varepsilon\int_{\mathbb{R}^n}\left(\psi(0,x)\, u_1(x)-\psi_s(0,x)\, u_0(x)\right)\mathrm{d}x \notag\\
& +\int_0^t\int_{\mathbb{R}^n}\left(\psi_{ss}(s,x)-\Delta\psi(s,x)\right)u(s,x)\,\mathrm{d}x\,\mathrm{d}s\notag\\
&=  \varepsilon \int_0^t \mathrm{e}^{-s/\beta } \int_{\mathbb{R}^n}\psi(s,x) \left( u_2(x)-\Delta u_0(x)\right)\mathrm{d}x\,\mathrm{d}s \notag\\
&\quad +\frac1\beta\int_0^t \int_0^s \mathrm{e}^{ (\tau-s)/\beta} \int_{\mathbb{R}^n}  \psi(s,x) |u(\tau,x)|^p\,\mathrm{d}x \,\mathrm{d}\tau \, \mathrm{d}s \label{Eq.Defn.Energy.Solution.Crit.Case}
\end{align} for any $\psi\in\mathcal{C}_0^{\infty}\big([0,T)\times\mathbb{R}^n\big)$ and any $t\in[0,T)$. 

We may now state the main result in the critical case for \eqref{Semi.MGT.Equation.epsilon}.

\begin{thm} \label{Thm critical case}
	Let $n\geqslant 2$ and $p=p_{\mathrm{Str}}(n) $.
	Let us assume  that $(u_0,u_1,u_2)\in H^2(\mb{R}^n)\times H^1(\mb{R}^n)\times L^2(\mb{R}^n)$ are nonnegative and compactly supported functions with supports contained in $B_R$ for some $R>0$ such that $u_0$ is not identically zero and $u_2-\Delta u_0$ is nonnegative. Let
	\begin{align*}
	u\in\mathcal{C}\big([0,T),H^2(\mathbb{R}^n)\big)\cap \mathcal{C}^1\big([0,T),H^1(\mathbb{R}^n)\big) \cap \mathcal{C}^2\big([0,T),L^2(\mathbb{R}^n)\big){\cap L^p_{\mathrm{loc}}([0,T)\times\mb{R}^n)}
	\end{align*}
	be an energy solution on $[0,T)$ to the Cauchy problem \eqref{Semi.MGT.Equation.epsilon} according to Definition \ref{Defn.Energy.Solution} with lifespan $T=T(\varepsilon)$ such that 
	\begin{align}
	\supp u(t,\cdot)\subset B_{R+t}\quad\mbox{for any}\,\,t\in(0,T). \label{Supp.u}
	\end{align}
	Then, there exists a positive constant $\varepsilon_0=\varepsilon_0(u_0,u_1,n,p,R,\beta)$ such that for any $\varepsilon\in(0,\varepsilon_0]$ the energy solution $u$ blows up in finite time. Moreover, the upper bound estimate for the lifespan
	\begin{align*}
	T(\varepsilon)\leqslant \exp \left( C\varepsilon^{-p(p-1)}\right)
	\end{align*}
	holds, where the constant $C>0$ is independent of $\varepsilon$.
\end{thm}

\subsection{Iteration frame and first lower bound estimate}

Before introducing the time -- dependent functional whose dynamic will be examined in order to prove Theorem \ref{Thm critical case}, let us prove a fundamental identity satisfied by local in time solutions to \eqref{Semi.MGT.Equation.epsilon}

\begin{prop} \label{Prop lower bounds critical case} Let $n\geqslant 2$ and $r>-1$. Assume that  $(u_0,u_1,u_2)\in H^2(\mb{R}^n)\times H^1(\mb{R}^n)\times L^2(\mb{R}^n)$ are compactly supported in $B_R$ for some $R>0$.  Let $u$ be an energy solution to \eqref{Semi.MGT.Equation.epsilon} on $[0,T)$ according to Definition \ref{Defn.Energy.Solution} satisfying \eqref{Supp.u}. Then, the following integral identity holds:
	\begin{align}
	\int_{\mathbb{R}^n}  u(t,x)\, \eta_{r}(t,t,x) \, \mathrm{d}x & = \varepsilon \int_{\mathbb{R}^n} u_0(x) \,\xi_{r}(t,x) \, \mathrm{d}x +  \varepsilon t \int_{\mathbb{R}^n}u_1(x)\, \eta_{r}(t,0,x) \, \mathrm{d}x\notag\\ 
	&  \quad +  \varepsilon \int_0^t (t-s)\, \mathrm{e}^{-s/\beta}\int_{\mathbb{R}^n}\left(u_2(x)-\Delta u_0(x)\right) \eta_{r}(t,s,x) \, \mathrm{d}x \,  \mathrm{d}s \notag \\
	& \quad + \frac1\beta \int_0^t (t-s)\int_0^s \mathrm{e}^{-(s-\sigma)/\beta} \!\int_{\mathbb{R}^n}|u(\sigma,x)|^p\, \eta_{r}(t,s,x) \, \mathrm{d}x \, \mathrm{d}\sigma\, \mathrm{d}s, \label{fund ineq G} 
	\end{align} for any $t\in (0,T)$, where $\xi_r$ and $\eta_r$ are defined in \eqref{def xi} and \eqref{def eta}, respectively.
\end{prop}

\begin{proof}
	Since  $u(t,\cdot)$ has compact support contained in $B_{R+t}$ for any $t\geqslant 0$, according to \eqref{Supp.u}, we can use the identity \eqref{Eq.Defn.Energy.Solution.Crit.Case} even for a noncompactly supported test function.
	Let $\Phi$ be the function defined by \eqref{def eigenfunction laplacian}.
	Since $\Phi$ satisfies $\Delta \Phi=\Phi$ and the function $y(t,s;\lambda)=\lambda^{-1} \sinh (\lambda(t-s))$ solves the parameter dependent ODE as follows: $$\left(\partial_s^2 -\lambda^2\right)y(t,s;\lambda)=0$$ with final conditions $y(t,t;\lambda)=0$ and $\partial_s y(t,t;\lambda)=-1$, then,  the test function $\psi(s,x)= y(t,s;\lambda) \,  \Phi(\lambda x)$ is a solution of the free wave equation $\psi_{ss}-\Delta \psi =0 $
	and, moreover, satisfies
	\begin{align*}
	& \psi(t,x)= 0 , \qquad \quad \ \qquad \psi(0,x)= \lambda^{-1} \sinh (\lambda t) \Phi(\lambda x), \\
	& \psi_s(t,x)= - \Phi(\lambda x) , \qquad \psi_s(0,x)= -\cosh (\lambda t) \Phi(\lambda x).
	\end{align*} 
	
	Let us begin to prove \eqref{fund ineq G}. Employing in \eqref{Eq.Defn.Energy.Solution.Crit.Case} the above defined test function $\psi$ and its properties, we obtain
	\begin{align*}
	\int_{\mathbb{R}^n} u(t,x)  \, \Phi(\lambda x) \, \mathrm{d}x  & = \varepsilon  \cosh (\lambda t) \int_{\mathbb{R}^n} u_0(x) \, \Phi(\lambda x) \, \mathrm{d}x + \varepsilon \frac{\sinh (\lambda t)}{\lambda} \int_{\mathbb{R}^n} u_1(x)\,  \Phi(\lambda x) \, \mathrm{d}x \\ & \quad +\varepsilon \int_0^t \frac{\sinh (\lambda(t-s))}{\lambda}  \,\mathrm{e}^{-s/\beta }  \int_{\mathbb{R}^n}\left( u_2(x)-\Delta u_0(x)\right) \Phi(\lambda x) \, \mathrm{d}x\,\mathrm{d}s \\ & \quad + \frac1\beta \int_0^t  \frac{\sinh (\lambda(t-s))}{\lambda} \int_0^s\mathrm{e}^{-(s-\sigma)/\beta}\int_{\mathbb{R}^n}|u(\sigma,x)|^p  \, \Phi(\lambda x) \, \mathrm{d}x \, \mathrm{d}\sigma \, \mathrm{d}s.
	\end{align*} Multiplying both sides of the above equality by $\mathrm{e}^{-\lambda(t+R)}\lambda^{r}$, integrating with respect to $\lambda$ over  the interval $[0,\lambda_0]$ and using Tonelli's theorem, we find \eqref{fund ineq G}.
\end{proof}

Hereafter until the end of Section  \ref{Section blow up critical case}, we shall assume that $u_0,u_1,u_2$ satisfy the assumptions from the statement of Theorem \ref{Thm critical case}, that is, these functions are nonnegative, compactly supported and satisfy $u_0\not \equiv 0$ and $u_2-\Delta u_0\geqslant 0$. Let $u$ be an energy solution of \eqref{Semi.MGT.Equation.epsilon} on $[0,T)$ according to Definition \ref{Defn.Energy.Solution}. We introduce the following time -- dependent functional:
\begin{equation}\label{defn functional crit case}
\begin{split}
\mathcal{U}(t)&\doteq \int_{\mathbb{R}^n} u(t,x) \, \eta_{r}(t,t,x) \, \mathrm{d}x ,
\end{split}
\end{equation} where $$r\doteq \frac{n-1}{2}-\frac{1}{p}.$$  From Proposition \ref{Prop lower bounds critical case} it follows immediately the positiveness of the functional $\mathcal{U}$, thanks to the assumptions on the Cauchy data.

The next step is to derive an integral inequalities involving $\mathcal{U}$ both in the left and in the right -- hand side, which will set the iteration frame for $\mathcal{U}$ in the iteration procedure.

\begin{prop} \label{prop integral ineq critic case} Let us assume that $r= (n-1)/2-1/p$. Let $\mathcal{U}$ be the functional defined by \eqref{defn functional crit case}. Then, there exist positive constants $C$ depending on $n,p,\beta,\lambda_0,R$ such that  the estimate
	\begin{align}
	\mathcal{U}(t) &\geqslant C \langle t\rangle^{-1} \int_0^t(t-s) \langle s \rangle^{-\frac{n-1}{2}+\frac{1}{p}} \int_0^s  \mathrm{e}^{-\frac{1}{\beta}(s-\sigma)}  \langle\sigma\rangle^{(n-1)(1-\frac p2)} \left(\log\langle\sigma\rangle\right)^{-(p-1)}  (\mathcal{U}(\sigma))^p \, \mathrm{d}\sigma\, \mathrm{d}s \label{Iteration Frame crit case}
	\end{align} holds for any $t\geqslant 0$.
\end{prop}

\begin{proof}  In the proof of this proposition we adapt the main ideas of Proposition 4.2 in \cite{WakYor18} to our model.
	By using H\"older's inequality and the support property for $u(\sigma,\cdot)$, we find
	\begin{align}
	\mathcal{U}(\sigma) \leqslant \left(\ \, \int_{\mathbb{R}^n}|u(\sigma,x)|^p \eta_{r}(t,s,x) \, \mathrm{d}x\right)^{\frac{1}{p}} \left(\ \int_{B_{\sigma+R}}\frac{\eta_{r}(\sigma,\sigma,x)^{p'}}{\eta_{r}(t,s,x)^{\frac{p'}{p}}}\, \mathrm{d}x\right)^{\frac{1}{p'}}. \label{holder + compact supp v}
	\end{align} We start by estimating the second factor on the right hand side in the last inequality. 
	
	According to our choice of $r$, both $r >(n-3)/2$ and $r>-1$ are always satisfied, so, by (ii) and (iii) in Lemma \ref{lemma eta and xi estimates}, since $|x|\leqslant \sigma+R$ implies $|x|\leqslant s+R$ for any $\sigma\in [0,s]$, we get
	\begin{align*}
	\int_{B_{\sigma+R}}\frac{\eta_{r}(\sigma,\sigma,x)^{p'}}{\eta_{r}(t,s,x)^{\frac{p'}{p}}}\, \mathrm{d}x & \lesssim \langle t\rangle^{\frac{p'}{p}}\langle s\rangle^{\frac{rp'}{p}} \langle \sigma\rangle^{-\frac{n-1}{2}p'}\int_{B_{\sigma+R}}\langle \sigma-|x|\rangle^{(\frac{n-3}{2}-r)p'} \, \mathrm{d}x \\
	& \lesssim  \langle t\rangle^{\frac{1}{p-1}}\langle s\rangle^{\frac{r}{p-1}} \langle \sigma\rangle^{-\frac{n-1}{2}p'}\int_{B_{\sigma+R}}\langle \sigma-|x|\rangle^{-1} \, \mathrm{d}x \\
	& \lesssim  \langle t\rangle^{\frac{1}{p-1}}\langle s\rangle^{\frac{r}{p-1}} \langle \sigma\rangle^{-\frac{n-1}{2}p'+n-1}\log\langle \sigma\rangle,
	\end{align*} where in the second step we used the definition of $r$ to get the exponent for the term in the $x$ -- dependent integral. Thanks to the sign assumptions on the Cauchy data, combining  \eqref{fund ineq G}, \eqref{holder + compact supp v} and the previous estimate, we arrive at
	\begin{align*}
	\mathcal{U}(t) &\gtrsim \int_0^t (t-s)\int_0^s \mathrm{e}^{-\frac{1}{\beta}(s-\sigma)}\int_{\mathbb{R}^n}|u(\sigma,x)|^p\, \eta_{r}(t,s,x) \, \mathrm{d}x \, \mathrm{d}\sigma\, \mathrm{d}s   \\ 
	& \gtrsim \int_0^t (t-s)\int_0^s  \mathrm{e}^{-\frac{1}{\beta}(s-\sigma)} \langle t\rangle^{-1}\langle s\rangle^{-r} \langle \sigma\rangle^{\frac{n-1}{2}p-(n-1)(p-1)} \left(\log\langle \sigma\rangle\right)^{-(p-1)}  (\mathcal{U}(\sigma))^p \, \mathrm{d}\sigma\, \mathrm{d}s \\
	&=  \langle t\rangle^{-1} \int_0^t (t-s)\langle s\rangle^{-\frac{n-1}{2}+\frac{1}{p}} \int_0^s \mathrm{e}^{-\frac{1}{\beta}(s-\sigma)} \langle \sigma\rangle^{\frac{n-1}{2}p-(n-1)(p-1)} \left(\log\langle \sigma\rangle\right)^{-(p-1)}  (\mathcal{U}(\sigma))^p \, \mathrm{d}\sigma\, \mathrm{d}s,
	\end{align*} which is exactly \eqref{Iteration Frame crit case}.
\end{proof}

\begin{lem} Let us suppose that the assumptions from Theorem \ref{Thm critical case} are fulfilled. Let $u$ be a solution to \eqref{Semi.MGT.Equation.epsilon} according to Definition \ref{Defn.Energy.Solution}. Then, there exists a positive constant $C_0=C_0(u_0,u_1,u_2,n,p,R)$ such that the following lower bound estimate:
	\begin{align}\label{lower bound int |u|^p no.2}
	\int_{\mathbb{R}^n}|u(t,x)|^p \, \mathrm{d}x \geqslant C_0 \varepsilon^p \langle t\rangle^{n-1-\frac{n-1}{2}p}
	\end{align}
	holds for any $t\geqslant 0$.
\end{lem}

\begin{proof}
	By using \eqref{fund ineq G} and the sign assumptions on initial data, we get 
	\begin{align} \label{intermediate estimate int |u|^p}
	\varepsilon \int_{\mathbb{R}^n} u_0(x) \, \xi_r(t,x) \, \mathrm{d}x\leqslant \mathcal{U}(t) \leqslant \left(\int_{\mathbb{R}^n}|u(t,x)|^p \, \mathrm{d}x \right)^{1/p} I(t)^{1/p'},
	\end{align} where we put
	\begin{align*}
	I(t) \doteq \int_{B_{t+R}}{\eta_r(t,t,x)^{p'}} \, \mathrm{d}x.
	\end{align*} Repeating exactly the same proof of Lemma 5.1 in \cite{WakYor18}, we have $$I(t)\lesssim \langle t\rangle^{n-1-\frac{n-1}{2}p'}.$$ 
	Combining this upper bound estimate for $I(t)$, Lemma \ref{lemma eta and xi estimates} (i) and \eqref{intermediate estimate int |u|^p}, it follows immediately our desired estimate \eqref{lower bound int |u|^p no.2}.
\end{proof}

\begin{rem} Let us point out explicitly that although \eqref{lower bound int |u|^p} and \eqref{lower bound int |u|^p no.2} are formally the same estimates, we had to prove this lower bound estimate twice as the energy solution $u$ satisfies different integral relations in the subcritical and critical case (cf. Definition \ref{Definition Energy solution} and Definition \ref{Defn.Energy.Solution}).
\end{rem}

\begin{prop}\label{Prop 1st lower bound log G} Let us assume that $r= (n-1)/2-1/p$ and $p=p_{\mathrm{Str}}(n)$. Let $\mathcal{U}$ be the functional defined by \eqref{defn functional crit case} and let us consider  a positive parameter $\omega_0>1$ such that $\beta \omega_0>1$. Then, there exist a positive constant $M$ depending on $n,p,\beta,\lambda_0,R,u_0,u_1$ such that  
	\begin{align}
	\mathcal{U}(t) &\geqslant M \varepsilon^p \log\left(\frac{t}{\beta \omega_0}\right) \label{1st logarithmic lower bound G}
	\end{align} holds for any $t\geqslant \beta \omega_0$.
\end{prop}

\begin{proof}  
	
	Combining the lower bound estimate \eqref{lower bound int |u|^p no.2} for the integral of the $p$ --  power of the solution $u$  together with Lemma \ref{lemma eta and xi estimates} (ii) and \eqref{fund ineq G}, we get
	\begin{align*}
	\mathcal{U}(t) & \gtrsim \int_0^t (t-s)\int_0^s  \mathrm{e}^{-\frac{1}{\beta}(s-\sigma)} \int_{\mathbb{R}^n}|u(\sigma,x)|^p\, \eta_{r}(t,s,x) \, \mathrm{d}x \, \mathrm{d}\sigma\, \mathrm{d}s \\
	& \gtrsim \langle t\rangle^{-1} \int_0^t (t-s)  \langle s \rangle^{-\frac{n-1}{2}+\frac{1}{p}}  \int_0^s  \mathrm{e}^{-\frac{1}{\beta}(s-\sigma)}\int_{\mathbb{R}^n}|u(\sigma,x)|^p\,  \, \mathrm{d}x \, \mathrm{d}\sigma\, \mathrm{d}s \\
	& \gtrsim \varepsilon^p \langle t\rangle^{-1} \int_0^t (t-s)  \langle s \rangle^{-\frac{n-1}{2}+\frac{1}{p}}  \int_0^s  \mathrm{e}^{-\frac{1}{\beta}(s-\sigma)}\langle \sigma\rangle^{n-1-\frac{n-1}{2}p} \, \mathrm{d}\sigma\, \mathrm{d}s,
	\end{align*} where we used again the sign assumptions on the Cauchy data. Therefore, for $t\geqslant 1$ by shrinking the domain of integration we find
	\begin{align*}
	\mathcal{U}(t) 
	& \gtrsim \varepsilon^p \langle t\rangle^{-1} \int_0^t (t-s)  \langle s \rangle^{-\frac{n-1}{2}p-\frac{n-1}{2}+\frac{1}{p}}  \int_{s/2}^s  \mathrm{e}^{-\frac{1}{\beta}(s-\sigma)}  \sigma^{n-1} \, \mathrm{d}\sigma\, \mathrm{d}s \\
	& \gtrsim \varepsilon^p \langle t\rangle^{-1} \int_0^t (t-s)  \langle s \rangle^{-\frac{n-1}{2}p-\frac{n-1}{2}+\frac{1}{p}}   s^{n-1}  \beta\left(1-\mathrm{e}^{-\frac{s}{2\beta}}\right) \mathrm{d}s \\
	& \gtrsim \varepsilon^p  \beta\left(1-\mathrm{e}^{-\frac{1}{2\beta}}\right) \langle t\rangle^{-1} \int_1^t (t-s)  \langle s \rangle^{-\frac{n-1}{2}p+\frac{n-1}{2}+\frac{1}{p}}  \, \mathrm{d}s .
	\end{align*} Due to $p=p_{\mathrm{Str}}(n)$, from \eqref{Eq.Stauss.Exponent}, we have
	\begin{align} \label{critical condition critical case}
	-\tfrac{n-1}{2}p+\tfrac{n-1}{2}+\tfrac{1}{p}=-1.
	\end{align} So, the power of $\langle s \rangle$ in the last integral is exactly $-1$. Hence, for $t\geqslant \beta \omega_0$ it results
	\begin{align*}
	\mathcal{U}(t) 
	& \gtrsim \varepsilon^p \langle t\rangle^{-1} \int_1^t (t-s)  \langle s \rangle^{-1}   \, \mathrm{d}s  \gtrsim \varepsilon^p \langle t\rangle^{-1} \int_1^t \frac{t-s}{s}    \, \mathrm{d}s \gtrsim \varepsilon^p \langle t\rangle^{-1} \int_{1}^t \log s   \, \mathrm{d}s \\ & \gtrsim \varepsilon^p \langle t\rangle^{-1} \int_{\tfrac{t}{\beta\omega_0}}^t \log s\, \mathrm{d}s \gtrsim \varepsilon^p  \log \left(\frac{t}{\beta \omega_0}\right) .
	\end{align*}  This completes the proof.
\end{proof}

In this subsection, we established the iteration frame \eqref{Iteration Frame crit case} for the functional $\mathcal{U}$ and we determined a first lower bound estimate \eqref{1st logarithmic lower bound G} for $\mathcal{U}$ containing a logarithmic factors. In the next subsection we are going to determine a sequence of lower bound estimates for $\mathcal{U}$ applying the slicing procedure. More specifically, we will combine the main ideas from \cite{PalTak19,PalTak19mix} concerning two step iteration procedures and from Section \ref{Section blow up} concerning the treatment of an exponential multiplier, in order to make the slicing procedure suitable for the iteration frame \eqref{Iteration Frame crit case}.

\subsection{Iteration argument via slicing method} \label{Subsection iteration method crit case}

Let us introduce the sequence $\{\omega_k\}_{k\in \mathbb{N}}$, where  $\omega_0$ has been introduced in the statement of Proposition \ref{Prop 1st lower bound log G} and $\omega_k\doteq 1+2^{-k}$ for any $k\geqslant 1$. Hence, the sequence of parameters that characterize the slicing procedure $\{\Omega_j\}_{j\in\mathbb{N}}$ is defined by 
\begin{align}\label{defintion Lj}
\Omega_j\doteq \prod_{k=0}^j \omega_k.
\end{align} We point out that $\{\Omega_j\}_{j\in\mathbb{N}}$ is an increasing sequence of positive real number and the infinite product $\prod_{k=0}^\infty \omega_k$ is convergent. Thus, if we denote $$\Omega\doteq \prod_{k=0}^\infty \omega_k,$$ then, in particular, $\Omega_j \uparrow \Omega$ as $j\to \infty$. 

The goal of this part is to prove via an iteration method the family of estimates
\begin{align}
\mathcal{U}(t)\geqslant M_j (\log\langle t\rangle)^{-b_j} \bigg(\log\bigg(\frac{t}{\beta \, \Omega_{2j}}\bigg)\bigg)^{a_j} 
\label{iter ineq crit case}
\end{align} for $t\geqslant \beta \, \Omega_{2j}$ and for any $j\in \mathbb{N}$, where $\{M_j\}_{j\in\mathbb{N}}$,  $\{a_j\}_{j\in\mathbb{N}}$ and  $\{b_j\}_{j\in\mathbb{N}}$ are sequences of nonnegative real numbers that will be determined recursively throughout the iteration procedure.
For $j=0$ we know that \eqref{iter ineq crit case} is true thanks to {Proposition \ref{Prop 1st lower bound log G}} with 
\begin{align*}
M_0\doteq M \varepsilon^{p},   \quad a_0\doteq 1 \quad \mbox{and} \quad b_0\doteq 0.
\end{align*}

We shall prove the validity of \eqref{iter ineq crit case} for any $j\in\mathbb{N}$ by induction. Since we have already shown the validity of the base case, it remains to prove the inductive step. Therefore, we assume that \eqref{iter ineq crit case} holds for $j\geqslant 1$ and we want to prove it for $j+1$. Plugging \eqref{iter ineq crit case} for $j$ in \eqref{Iteration Frame crit case}, we find 
\begin{align*}
\mathcal{U}(t) & \geqslant C M_j^p \langle t\rangle^{-1} \!\!\int_{\beta\Omega_{2j}}^t(t-s) \langle s \rangle^{-r}\!\!\int_{\beta\Omega_{2j}}^s \mathrm{e}^{-\frac1\beta (s-\sigma)} \langle\sigma\rangle^{(n-1)(1-\frac{p}{2})}\\
&\qquad\qquad\qquad \times\left(\log\langle\sigma\rangle\right)^{-(p-1)-b_jp}  \left(\log\left(\frac{\sigma}{\beta\Omega_{2j}}\right)\right)^{a_j p}  \mathrm{d}\sigma\, \mathrm{d}s \\
&  \geqslant C M_j^p \langle t\rangle^{-1} \left(\log\langle t \rangle\right)^{-(p-1)-b_jp}  \!\!\int_{\beta\Omega_{2j}}^t(t-s) \langle s \rangle^{-r-\frac{n-1}{2}p}\!\!\int_{\beta\Omega_{2j}}^s \mathrm{e}^{-\frac1\beta (s-\sigma)}\\
&\qquad\qquad\qquad\qquad\qquad\qquad\qquad \times\langle\sigma\rangle^{n-1}  \left(\log\left(\frac{\sigma}{\beta\Omega_{2j}}\right)\right)^{a_j p}  \mathrm{d}\sigma\, \mathrm{d}s
\end{align*} for $t\geqslant \beta \,\Omega_{2j+2}$. Now we estimate  for $s\geqslant \beta\, \Omega_{2j+1}$ from below the $\sigma$ -- integral in the last line as follows:
\begin{align*}
\int_{\beta\Omega_{2j}}^s \mathrm{e}^{-\frac1\beta (s-\sigma)} \langle\sigma\rangle^{n-1}  & \left(\log\left(\frac{\sigma}{\beta\Omega_{2j}}\right)\right)^{a_j p}  \mathrm{d}\sigma  \geqslant \int_{\tfrac{\Omega_{2j} \, s}{\Omega_{2j+1}} }^s \mathrm{e}^{-\frac1\beta (s-\sigma)} \sigma^{n-1}   \left(\log\left(\frac{\sigma}{\beta\Omega_{2j}}\right)\right)^{a_j p}  \mathrm{d}\sigma \\
& \geqslant \left(\frac{\Omega_{2j}}{\Omega_{2j+1}}\right)^{n-1} s^{n-1}  \left(\log\left(\frac{s}{\beta\Omega_{2j+1}}\right)\right)^{a_j p} \int_{\tfrac{\Omega_{2j} \, s}{\Omega_{2j+1}} }^s \mathrm{e}^{-\frac1\beta (s-\sigma)} \, \mathrm{d}\sigma
\\
& \geqslant \beta \left(\frac{\Omega_{2j}}{\Omega_{2j+1}}\right)^{n-1} \left(1-\mathrm{e}^{-\frac1\beta \left(1-\tfrac{\Omega_{2j}}{\Omega_{2j+1}}\right)s} \right) s^{n-1}  \left(\log\left(\frac{s}{\beta\Omega_{2j+1}}\right)\right)^{a_j p} 
\\
& \geqslant \beta \left(\frac{\Omega_{2j}}{\Omega_{2j+1}}\right)^{n-1} \left(1-\mathrm{e}^{-\left(\Omega_{2j+1}-\Omega_{2j}\right)} \right) s^{n-1}  \left(\log\left(\frac{s}{\beta\Omega_{2j+1}}\right)\right)^{a_j p}.
\end{align*} Using the inequalities $\Omega_{2j}/\Omega_{2j+1} = {1/\omega_{2j+1}}> 1/2$, the estimate $4s\geqslant \langle s\rangle$ for any $s\geqslant 1$ and the inequality
\begin{align*}
1-\mathrm{e}^{-\left(\Omega_{2j+1}-\Omega_{2j}\right)} &= 1-\mathrm{e}^{-\Omega_{2j}\left(\omega_{2j+1}-1 \right)}\geqslant 1-\mathrm{e}^{-\left(\omega_{2j+1}-1 \right)}\geqslant 1 -\left(1- \left(\omega_{2j+1}-1 \right)+\tfrac12 \left(\omega_{2j+1}-1 \right)^2\right) \\
& \geqslant \left(\omega_{2j+1}-1 \right)\left(1-\tfrac12 \left(\omega_{2j+1}-1 \right)\right) = 2^{-2(2j+1)} \left(2^{2j+1}-\tfrac12 \right) \geqslant 2^{-2(2j+1)},
\end{align*} where we used $\Omega_{2j}>1$ and Taylor's formula up to order 2 neglecting the positive remaining term, we obtain 
\begin{align*}
\int_{\beta\Omega_{2j}}^s \mathrm{e}^{-\frac1\beta (s-\sigma)} \langle\sigma\rangle^{n-1}  \left(\log\left(\frac{\sigma}{\beta\Omega_{2j}}\right)\right)^{a_j p}  \mathrm{d}\sigma  & \geqslant    2^{-3(n-1)} \beta \, 2^{-2(2j+1)} \langle s\rangle^{n-1}  \left(\log\left(\frac{s}{\beta\Omega_{2j+1}}\right)\right)^{a_j p}.
\end{align*}
So, plugging the lower bound estimate for the $\sigma$ -- integral in the lower bound estimate for $\mathcal{U}(t)$ and using again \eqref{critical condition critical case}, for $t\geqslant \beta \, \Omega_{2j+2}$ it holds
\begin{align*}
\mathcal{U}(t) & \geqslant \widehat{C}\, 2^{-2(2j+1)}  M_j^p \left(\log\langle t\rangle\right)^{-(p-1)-b_jp}  \langle t\rangle^{-1} \int_{\beta \Omega_{2j+1}}^t(t-s) \langle s \rangle^{-\frac{n-1}{2}p+\frac{n-1}{2}+\frac{1}{p}}  \left(\log\left(\frac{s}{\beta\Omega_{2j+1}}\right)\right)^{a_j p}\, \mathrm{d}s \\
& \geqslant \widehat{C} \, 2^{-2(2j+1)}  M_j^p \left(\log\langle t\rangle\right)^{-(p-1)-b_jp}  \langle t\rangle^{-1} \int_{\beta \Omega_{2j+1}}^t(t-s) \langle s \rangle^{-1}  \left(\log\left(\frac{s}{\beta\Omega_{2j+1}}\right)\right)^{a_j p}\, \mathrm{d}s
\\
& \geqslant  \widehat{C} \, 2^{-2(2j+2)}  M_j^p \left(\log\langle t\rangle\right)^{-(p-1)-b_jp}  \langle t\rangle^{-1} \int_{\beta \Omega_{2j+1}}^t \frac{t-s}{s}  \left(\log\left(\frac{s}{\beta\Omega_{2j+1}}\right)\right)^{a_j p}\, \mathrm{d}s,
\end{align*} where $\widehat{C}\doteq 2^{-3(n-1)} \beta C$. Let us estimate the $s$ -- dependent integral in the last  line. Integration by parts  and a further shrinking of the domain of integration lead to
\begin{align*}
\langle t\rangle^{-1}\int_{\beta \Omega_{2j+1}}^t \frac{t-s}{s}  \left(\log\left(\frac{s}{\beta\Omega_{2j+1}}\right)\right)^{a_j p}\, \mathrm{d}s & = (a_j p+1)^{-1} \langle t\rangle^{-1} \int_{\beta \Omega_{2j+1}}^t \left(\log\left(\frac{s}{\beta\Omega_{2j+1}}\right)\right)^{a_j p+1}\, \mathrm{d}s \\
& \geqslant (a_j p+1)^{-1} \langle t\rangle^{-1} \int_{\tfrac{\Omega_{2j+1} t}{ \Omega_{2j+2}}}^t \left(\log\left(\frac{s}{\beta\Omega_{2j+1}}\right)\right)^{a_j p+1}\, \mathrm{d}s 
\\
& \geqslant (a_j p+1)^{-1} \langle t\rangle^{-1} \left(\log\left(\frac{t}{\beta\Omega_{2j+2}}\right)\right)^{a_j p+1}  \left( 1-\frac{\Omega_{2j+1} }{ \Omega_{2j+2}}\right) t 
\\
&  \geqslant 2^{-3} 2^{-(2j+2)} (a_j p+1)^{-1} \left(\log\left(\frac{t}{\beta\Omega_{2j+2}}\right)\right)^{a_j p+1} 
\end{align*} for $t\geqslant \beta\, \Omega_{2j+2}$, where in the last estimate we used $4 t\geqslant \langle t\rangle $ and 
\begin{align*}
1-\frac{\Omega_{2j+1} }{ \Omega_{2j+2}} ={\frac{\omega_{2j+2}-1}{\omega_{2j+2}}}\geqslant 2^{-1} (\omega_{2j+2}-1) = 2^{-1-(2j+2)}.
\end{align*}
Therefore, combining this time the lower bound estimate for $\mathcal{U}(t)$ with the estimate from below of the $s$ -- integral, we conclude
\begin{align*}
\mathcal{U}(t) & \geqslant 2^{-3} \widehat{C} \, (a_j p+1) ^{-1} 2^{-3(2j+2)} M_j^p \left(\log\langle t\rangle\right)^{-(p-1)-b_jp}  \left(\log\left(\frac{t}{\beta\Omega_{2j+2}}\right)\right)^{a_j p+1}  
\end{align*} for $t\geqslant {\beta\, \Omega_{2j+2}}$. Also, we proved \eqref{iter ineq crit case} for $j+1$ provided that 
\begin{align*}
M_{j+1} \doteq 2^{-3n} \beta C \, (a_j p+1) ^{-1} 2^{-6(j+1)} M_j^p, \quad a_{j+1} \doteq a_jp+1, \quad b_{j+1}\doteq (p-1)+b_jp.
\end{align*}
As next step, we determine a lower bound for the term $M_j$ which is more easy to handle. For this reason, we determine an explicit representation of the exponents $a_j$ and $b_j$. By using recursively the relations $a_j=1+pa_{j-1}$ and $b_j= (p-1) + p b_{j-1}$ and the initial exponents $a_0=1$, $b_0=0$, we get 
\begin{align}\label{explit expressions aj and bj}
a_j & = a_0 p^j +\sum_{k=0}^{j-1} p^k =  \tfrac{p^{j+1}-1}{p-1} \quad \mbox{and} \quad
b_j = p^j b_0 +(p-1) \sum_{k=0}^{j-1} p^k = p^j-1.
\end{align} In particular, $a_{j-1}p+1= a_j\leqslant p^{j+1}/(p-1) $ implies that 
\begin{align} \label{lower bound Mj no.1}
M_j \geqslant \widehat{D} \, (2^{6}p)^{-j} M^p_{j-1}
\end{align} for any $j\geqslant 1$, where $\widehat{D}\doteq 2^{-3n}\beta C (p-1)/p$. Applying the logarithmic function to both sides of  \eqref{lower bound Mj no.1} and using iteratively the resulting inequality, we have
\begin{align*}
\log M_j & \geqslant p \log M_{j-1} -j \log \big(2^{6}p\big)+\log \widehat{D} \\
& \geqslant p^j \log M_0 -\Bigg(\sum_{k=0}^{j-1}(j-k)p^k \Bigg)\log \big(2^{6}p\big)+\Bigg(\sum_{k=0}^{j-1} p^k \Bigg)\log \widehat{D}\\
& = p^j \!\left(\log M_0-\frac{p\log \big(2^{6}p\big)}{(p-1)^2}+\frac{\log\widehat{D}}{p-1}\right)+\left( \frac{j}{p-1}+\frac{p}{(p-1)^2}\right)\log \big(2^{6}p\big)-\frac{\log\widehat{D}}{p-1},
\end{align*} where we used again \eqref{identity sum (j-k)p^k}. Let us define {$j_1=j_1(n,p,\beta)$} as the smallest nonnegative integer such that $$j_1\geqslant \frac{\log \widehat{D}}{\log \big(2^{6}p\big)}-\frac{p}{p-1}.$$ Then, for any $j\geqslant j_1$ we may estimate 
\begin{align} \label{lower bound Mj no.2}
\log M_j & \geqslant  p^j \left(\log M_0-\frac{p\log \big(2^{6}p\big)}{(p-1)^2}+\frac{\log\widehat{D}}{p-1}\right) = p^j \log (N_0 \varepsilon^p),
\end{align} where $N_0\doteq M \big(2^{6}p\big)^{-p/(p-1)^2}\widehat{D}^{1/(p-1)}$. Consequently, combining \eqref{iter ineq crit case}, \eqref{explit expressions aj and bj} and \eqref{lower bound Mj no.2}, it results
\begin{align*}
\mathcal{U}(t)&\geqslant \exp \left( p^j\log(N_0\varepsilon^p)\right) \left(\log\langle t\rangle \right)^{-p^j+1} \left(\log \left(\frac{t}{\beta\Omega}\right)\right)^{(p^{j+1}-1)/(p-1)} \\
&= \exp \left( p^j\log\left(N_0\varepsilon^p \left(\log\langle t\rangle\right)^{-1}\left(\log \left(\frac{t}{\beta\Omega}\right)\right)^{p/(p-1)}\right) \right) \log\langle t\rangle  \left(\log \left(\frac{t}{\beta\Omega}\right)\right)^{-1/(p-1)}
\end{align*} for $t\geqslant \beta \Omega$ and any $j\geqslant j_1$.

For $t\geqslant t_0\doteq {\max\left\{3,(\beta\Omega)^{\frac{\beta\Omega}{\beta\Omega-1}}\right\}} $ the inequalities $$\log\langle t\rangle \leqslant \log (2t) \leqslant 2\log t \quad \mbox{and} \quad \log \left(\frac{t}{\beta\Omega}\right)\geqslant \frac{1}{\beta \Omega} \log t $$ hold true, so,
\begin{align}
\mathcal{U}(t)&\geqslant  \exp \left( p^j\log\left(N_1\varepsilon^p \left(\log t\right)^{1/(p-1)}\right) \right) \log\langle t\rangle  \left(\log \left(\frac{t}{\beta\Omega}\right)\right)^{-1/(p-1)} \label{final lower bound G}
\end{align} for $t\geqslant t_0$ and any $j\geqslant j_1$, where $N_1\doteq 2^{-1} (\beta\Omega)^{ -p/(p-1)}N_0$.
Let us introduce the function $$J(t,\varepsilon)\doteq N_1\varepsilon^p \left(\log t\right)^{1/(p-1)}.$$ We can choose {$\varepsilon_0=\varepsilon_0(n,p,\beta,\lambda_0,R,u_0,u_1)$} sufficiently small so that $$\exp \left(N_1^{1-p}\varepsilon_0^{-p(p-1)}\right)\geqslant t_0.$$ Hence, for any $\varepsilon\in (0,\varepsilon_0]$ and for $t> \exp \left(N_1^{1-p}\varepsilon^{-p(p-1)}\right)$ the two conditions $t\geqslant t_0$ and $J(t,\varepsilon)>1$ are always fulfilled. Consequently, for any $\varepsilon\in (0,\varepsilon_0]$ and for $t> \exp \left(N_1^{1-p}\varepsilon^{-p(p-1)}\right)$ taking the limit as $j\to \infty$ in \eqref{final lower bound G} we find that the lower bound for $\mathcal{U}(t)$ blows up. Then, $\mathcal{U}(t)$ cannot be finite for this $t$. Thus, we proved that $\mathcal{U}$ blows up in finite time and, furthermore, we provided the upper bound estimate for the lifespan $$T(\varepsilon)\leqslant \exp \left(N_1^{1-p}\varepsilon^{-p(p-1)}\right).$$ This completes the proof of Theorem \ref{Thm critical case}.

\section{Concluding remarks}\label{Section Concluding remarks}

Let us consider the general case of the Cauchy problem for the semilinear MGT equation
\begin{equation}\label{Semi.MGT.Equation.General}
\begin{cases}
\tau u_{ttt}+u_{tt}-\Delta u-\beta \Delta u_t=|u|^p,&x\in\mb{R}^n,\,t>0,\\
(u,u_t,u_{tt})(0,x)=(u_0,u_1,u_2)(x),&x\in\mb{R}^n,
\end{cases}
\end{equation}
where $0<\tau<\beta$ (the dissipative case) or $\tau=\beta$ (the conservative case) and $p>1$. By introducing
\begin{align*}
w \doteq \tau u_t+u,
\end{align*}
we may transform \eqref{Semi.MGT.Equation.General} to the following semilinear second order evolution equation:
\begin{equation}\label{Semi.MGT.Equation.General.Transfer}
\begin{cases}
w_{tt}-\tfrac{\beta}{\tau}\Delta w+\tfrac{\beta-\tau}{\tau^2}  G\ast\Delta w=H(u_0,w;p,\beta,\tau),&x\in\mb{R}^n,\,t>0,\\
(w,w_t)(0,x)=(\tau u_1+u_0,\tau u_2+u_1)(x),&x\in\mb{R}^n,
\end{cases}
\end{equation}
where the kernel function $G=G(t)$ is given by $G(t)\doteq \mathrm{e}^{-t/\tau}$ and the right -- hand side is
\begin{align*}
H(u_0,w;p,\beta,\tau)(t,x) \doteq-\tfrac{\beta-\tau}{\tau}\mathrm{e}^{-t/\tau}\Delta u_0(x)+\left|\mathrm{e}^{-t/\tau}u_0(x)+\tfrac{1}{\tau}(G\ast w)(t,x)\right|^p.
\end{align*}
Here the convolution term is defined by
\begin{align*}
(G\ast w)(t,x) \doteq \int_0^tG(t-s) w(s,x)\mathrm{d}s.
\end{align*}
We now may understand the equation in \eqref{Semi.MGT.Equation.General.Transfer} in the two cases above mentioned. In the conservative case $\tau=\beta$, we may interpret the model \eqref{Semi.MGT.Equation.General.Transfer} as a wave equation with power source nonlinearity, which includes a memory term with an exponential decaying kernel function. On the other hand, in the dissipative case $0<\tau<\beta$, the dissipation generated by the memory term comes into play even in the linear part, therefore, the model \eqref{Semi.MGT.Equation.General.Transfer} can be interpreted as a semilinear viscoelastic equation (see for example \cite{DhRiveraKawashima,RackeSaidH2012}). Up to the knowledge of the authors, the blow -- up of solutions to this kind of semilinear viscoelastic equations is still an open problem.

In this paper, we considered the Cauchy problem for the semilinear MGT equation with power nonlinearity $|u|^p$ and proved the blow -- up of local energy solutions in the sub -- Strauss case, i.e.,  for $1<p\leqslant p_{\mathrm{Str}}(n)$. Concerning the Cauchy problem for the semilinear MGT equation with nonlinearity of derivative type in the conservative case, namely,
\begin{equation}\label{Semi.MGT.Equation.u_t}
\begin{cases}
\beta u_{ttt}+u_{tt}-\Delta u-\beta \Delta u_t=|u_t|^p,&x\in\mb{R}^n,\,t>0,\\
(u,u_t,u_{tt})(0,x)=(u_0,u_1,u_2)(x),&x\in\mb{R}^n,
\end{cases}
\end{equation}
with $\beta>0$, in the forthcoming paper \cite{ChenPalmieri201902}, we shall study the blow -- up of local in time solutions to \eqref{Semi.MGT.Equation.u_t} and the corresponding lifespan estimates  under suitable assumptions for initial data. More specifically, the blow -- up in finite time of energy solutions to  \eqref{Semi.MGT.Equation.u_t} is going to be proved providing that the power $p$ of the nonlinearity satisfies
\begin{align*}
1<p\leqslant p_{\mathrm{Gla}}(n)\doteq\frac{n+1}{n-1},
\end{align*} 
for $n\geqslant2$ and $p>1$ for $n=1$. We underline that the Glassey exponent $p_{\mathrm{Gla}}(n)$ is the critical exponent for the corresponding semilinear wave equation with nonlinearity of derivative type.
\vspace{0.25cm}

\paragraph*{Acknowledgments} The Ph.D. study of the first author is supported by S\"achsiches Landesgraduiertenstipendium. The second author is supported by the University of Pisa, Project PRA 2018 49.

\end{document}